\documentclass[11pt,a4paper]{article}
\usepackage{amssymb,amsmath,amsopn,amsthm,graphicx,makeidx}
\usepackage[all,2cell]{xy} \UseAllTwocells

\begin{document}

\newtheorem{definition}{Definition}[section]
\newtheorem{definitions}[definition]{Definitions}
\newtheorem{lemma}[definition]{Lemma}
\newtheorem{prop}[definition]{Proposition}
\newtheorem{theorem}[definition]{Theorem}
\newtheorem{cor}[definition]{Corollary}
\newtheorem{cors}[definition]{Corollaries}
\theoremstyle{remark}
\newtheorem{remark}[definition]{Remark}
\theoremstyle{remark}
\newtheorem{remarks}[definition]{Remarks}
\theoremstyle{remark}
\newtheorem{notation}[definition]{Notation}
\theoremstyle{remark}
\newtheorem{example}[definition]{Example}
\theoremstyle{remark}
\newtheorem{examples}[definition]{Examples}
\theoremstyle{remark}
\newtheorem{dgram}[definition]{Diagram}
\theoremstyle{definition}
\newtheorem{fact}[definition]{Fact}
\theoremstyle{remark}
\newtheorem{illust}[definition]{Illustration}
\theoremstyle{remark}
\newtheorem{rmk}[definition]{Remark}
\theoremstyle{definition}
\newtheorem{question}[definition]{Question}
\theoremstyle{definition}
\newtheorem{conj}[definition]{Conjecture}

\renewcommand{\marginpar}[2][]{}

\newcommand{\stac}[2]{\genfrac{}{}{0pt}{}{#1}{#2}}
\newcommand{\stacc}[3]{\stac{\stac{\stac{}{#1}}{#2}}{\stac{}{#3}}}
\newcommand{\staccc}[4]{\stac{\stac{#1}{#2}}{\stac{#3}{#4}}}
\newcommand{\stacccc}[5]{\stac{\stacc{#1}{#2}{#3}}{\stac{#4}{#5}}}

\renewenvironment{proof}{\noindent {\bf{Proof.}}}{\hspace*{3mm}{$\Box$}{\vspace{9pt}}}

\title{Modules with irrational slope over tubular algebras}
\author{Richard Harland\footnote{richard.harland@gmail.com} and Mike Prest\footnote{School of Mathematics, Alan Turing Building, University of Manchester, Manchester M13 9PL, UK, mprest@manchester.ac.uk}}

\maketitle

\abstract{Let $A$ be a tubular algebra and let $r$ be a positive irrational. Let ${\mathcal D}_r$ be the definable subcategory of $A$-modules of slope $r$. Then the width of the lattice of pp formulas for ${\mathcal D}_r$ is $\infty$. It follows that if $A$ is countable then there is a superdecomposable pure-injective module of slope $r$.\footnote{16G20, 16G70, 03C60; tubular algebra, pure-injective module, definable subcategory, pp formula, width, superdecomposable module}}

\tableofcontents

\section{Introduction}\label{intro}

The tubular algebras constitute a particular class of finite-dimensional algebras of global dimension 2, see \cite{RinTame}; their representation theory has a geometric interpretation, see e.g.~\cite{GeiLenz}, and there is continuing interest in understanding their representations, see e.g.~\cite{BarSchr}, \cite{ReiRin}.  The finite-dimensional modules over these algebras were described in \cite[Chpt.~5]{RinTame}.  There it was shown that the indecomposable finite-dimensional modules fall into families of Auslander-Reiten components, these families being parametrised by ${\mathbb Q}^\infty_0$, the non-negative rationals augmented by a maximal element $\infty$; moreover all non-zero morphisms go from left to right; see Theorem \ref{canonmodcat} for the precise statement.

Here we are interested in the infinite-dimensional representation theory of tubular algebras.  The theorem referred to above associates to each finite-dimensional indecomposable module (with just a few exceptions at the ``extremes" of the category of modules) an ``index" or  ``slope" which is a non-negative rational or $\infty$.   It had been observed, e.g.~in \cite{PreMor}, that if the underlying field is countable then there are uncountably many indecomposable pure-injective modules, indeed (e.g., \cite{PreCanon}) for each irrational cut, $r$, of ${\mathbb Q}^\infty_0$ there is at least one indecomposable, necessarily infinite-dimensional, pure-injective module of slope $r$; that result is considerably strengthened here.

It is a remarkable theorem of Reiten and Ringel \cite{ReiRin} that, over these algebras, every infinite-dimensional indecomposable module (whether pure-injective or not) has a slope.  None of the indecomposable pure-injective modules with irrational slope has been explicitly described as yet but it is shown here that there are a great many of these of a given slope, at least if the underlying field $k$ is countable. In that case, for each irrational $r$ there are continuum many indecomposable pure-injective modules of slope $r$.  We remark that finite-dimensional modules are pure-injective and, in general, the pure-injective modules form what seems to be a tractable class of modules with properties strongly related to those of the category of finite-dimensional modules (see, e.g., \cite{PreNBK}, \cite{JeLe}).

What we actually prove is that if ${\mathcal D}_r$ is the definable subcategory of modules of slope $r$ then the width, in the sense of Ziegler \cite{Zie}, of the corresponding lattice of pp-definable subgroups is undefined.  This is a property that can be seen in the category of finite-dimensional modules, see the comments after \ref{KGatr}.  It follows in the case that $k$ is countable that there is, for each positive irrational $r$, a superdecomposable pure-injective module - one without any indecomposable direct summand - of slope $r$, as well as $2^{\aleph_0}$ many indecomposable pure-injectives of slope $r$.  (Countability of $k$ is a current restriction, in that the general implication between the lattice of pp-definable subgroups having undefined width and the existence of superdecomposable pure-injectives, as well as many indecomposable pure-injectives, is, in one direction, proved only if the field is countable.)

The results in this paper are mostly due to the first author and are contained in his doctoral thesis \cite{Har}, the research for which was supported by a EPSRC DTA scholarship and supervised by the second author.  The latter prepared this paper for publication and, while doing so, extended some of the results (mainly in Section \ref{secfurther}).

\section{Background}

We have to assume a certain amount of background from both representation theory of finite-dimensional algebras and the model theory of modules.  Here we recall some definitions and state results that we will use but, for more details, the reader should consult the references we quote.  We assume throughout that the base field $k$ is algebraically closed and that all algebras appearing are finite-dimensional $k$-algebras; when it matters they should be assumed to be basic and connected.

\subsection{Tubular algebras}

We refer to \cite{RinTame} or \cite{AssSimSkow1, SimSkow2, SimSkow3} for most background and just recall what we will need here.  If $A$ is a finite-dimensional $k$-algebra then we write $A\mbox{-}{\rm Mod}$, $A\mbox{-}{\rm mod}$, $A\mbox{-}{\rm ind}$ for, respectively, the category of left $A$-modules, the category of finite-dimensional left $A$-modules and the set of isomorphism classes of indecomposable finite-dimensional $A$-modules.

By ${\rm add}({\mathcal X})$ we denote the closure of the class ${\mathcal X}$ of modules under finite direct sums and direct summands.  We also write, for instance, $(M,{\mathcal X})=0$ to mean $(M,X)=0$ for all $X\in {\mathcal X}$.  (Usually we abbreviate ${\rm Hom}_R(M,X)$ to $(M,X)$.)

If the simple $A$-modules are $S_1,\dots, S_n$ with corresponding projective covers $P_1,\dots , P_n$ then the {\bf dimension vector} of a module $M$ is given by $\underline{\rm dim} (M) =({\rm dim}(P_1,M),\dots ,{\rm dim}(P_n,M))$.  Note that ${\rm dim}(P_i,M)$ equals the number of occurrences of $S_i$ in a composition series of $M$ and $\underline{\rm dim}$ gives an isomorphism of the Grothendieck group ${\rm K}_0(A)$ with ${\mathbb Z}^n$.

We will use the standard bilinear form $\langle -,- \rangle$ defined on ${\rm K}_0(A)$ and the fact that, if $A$ has finite global dimension, then $\langle \underline{\rm dim}(M), \underline{\rm dim}(N)\rangle =\sum_{i=0}^{\infty} (-1)^i {\rm dim}\, {\rm Ext}^i(M,N)$ (see, e.g., \cite[III.3.13]{AssSimSkow1}).  In the cases that we look at, either $M$ will have projective dimension 1 or $N$ will have injective dimension 1, so only the ${\rm Hom}$ and ${\rm Ext}^1$ terms will be non-zero.  The corresponding quadratic form is given by $\chi_A(x) =\langle x, x\rangle$ and the {\bf radical} of $\chi_A$ is the subgroup ${\rm rad}(\chi_A) =\{ x: \chi_A(x) =0\}$ of ${\rm K}_0(A)$; its elements are the {\bf radical} vectors.

We recall the following special case of the Auslander-Reiten formula (see, e.g., \cite[IV, 2.15]{AssSimSkow1}).

\begin{theorem}\label{AusReit}\marginpar{AusReit} If $A$ is a finite-dimensional $K$-algebra and $M,N\in {\rm mod}\mbox{-}A$ with $M$ of projective dimension $\leq 1$ then ${\rm Ext}^1(M,N)\simeq (N,\tau M)$ as vector spaces; if the injective dimension of $N$ is $\leq 1$ then ${\rm Ext}^1(M,N) \simeq (\tau^{-1}N, M)$, where $\tau$ denotes Auslander-Reiten translate.
\end{theorem}

We don't give the, rather technical, definition of tubular algebra here (there are examples in Section \ref{sechomog}) - for that see \cite[Chapter 5]{RinTame} or \cite[XIX, 3.19]{SimSkow3}; these include the canonical tubular algebras which are directly defined in terms of their quivers, see \cite[\S 3.7]{RinTame} or \cite[XX, 3.14]{SimSkow3}.  Most of what we need about the representation theory of these algebras is recalled below.

\begin{theorem}\label{canonmodcat}\marginpar{canonmodcat} \cite[\S 5.2]{RinTame} Let $A$ be a tubular algebra; then $A\mbox{-}{\rm ind} = {\mathcal P}_0 \, \cup \, \bigcup \{ {\mathcal T}_q: q \in {\mathbb Q}_0^\infty\} \cup {\mathcal Q}_\infty$ (disjoint union) where each ${\mathcal T}_q$ is a tubular family separating ${\mathcal P}_q ={\mathcal P}_0 \,  \cup \bigcup \, \{ {\mathcal T}_{q'}: q'<q\}$ from ${\mathcal Q}_q = \bigcup \{ {\mathcal T}_{q'}: q< q' \} \cup {\mathcal Q}_\infty$.
\end{theorem}

This means that $0= ({\mathcal T}_q, {\mathcal P}_q) = ({\mathcal Q}_q, {\mathcal T}_q) = ({\mathcal Q}_q, {\mathcal P}_q)$, that $({\mathcal T}(\rho), {\mathcal T}(\rho'))=0$ for any distinct tubes ${\mathcal T}(\rho), {\mathcal T}(\rho')$ in ${\mathcal T}_q$ and that every map from a module in ${\mathcal P}_q$ to one in ${\mathcal Q}_q$ factors through ${\rm add}({\mathcal T}(\rho))$ where ${\mathcal T}(\rho)$ is any tube in ${\mathcal T}_q$.  We also use the following.

\begin{fact}\label{pdid1}\marginpar{pdid1} (\cite[3.1.5]{RinTame}) Every module in $\bigcup \{ {\mathcal T}_q: q \in {\mathbb Q}^+\}$ has both injective and projective dimension $1$.
\end{fact}

If $A$ is a tubular algebra then there is (\cite[\S 5.1]{RinTame}) a canonical pair $h_0, h_\infty$ of linearly independent radical vectors.  These generate a subgroup of ${\rm rad}(\chi_A)$ of finite index.  Also $\langle h_\infty,h_0\rangle = - \langle h_0,h_\infty \rangle $. For any vector $x$ define the {\bf index} of $M\in {\rm ind}\mbox{-}A$ to be the ratio $\iota(\underline{\rm dim}(M)) =- \dfrac{\langle h_0,\underline{\rm dim}(M)\rangle}{\langle h_\infty, \underline{\rm dim}(M) \rangle}$; then, if $M$ is in neither ${\mathcal P}_0$ nor ${\mathcal Q}_\infty$, we have $\iota(\underline{\rm dim}(M) =q$ iff $M\in {\mathcal T}_q$.

There is a more general notion, not confined to finitely generated modules:  we say that the {\bf slope} of a module $M$ is $r\in {\mathbb R}_0^\infty$ if $(M,{\mathcal P}_r)=0 =({\mathcal Q}_r, M)$ where ${\mathcal P}_r$ and ${\mathcal Q}_r$ are defined as in \ref{canonmodcat} but with $r$ in place of the rational $q$.  By \ref{AusReit} and the fact that ${\mathcal P}_r$ and ${\mathcal Q}_r$, being unions of Auslander-Reiten components, are closed under $\tau^{\pm 1}$, this is equivalent to ${\rm Ext}^1({\mathcal P}_r, M) =0 = ({\mathcal Q}_r,M)$.  That is clear for finite-dimensional $M$; for general $M$ we can argue, for example, as follows.  First suppose that $(M,{\mathcal P}_r)=0$; take $N\in {\mathcal P}_r$, of slope $q$ say.  By \ref{ReiRinLem} below, $M$ is a direct limit of submodules $M'$ of slope $>q$.  For each such $M'$ we have ${\rm Ext}^1(N,M')\simeq (M',\tau N)=0$ so, since ${\rm Ext}^1(N,-)$ commutes with direct limits (\cite[Thm.~2]{Bro2}), we deduce ${\rm Ext}^1(N,M)=0$.  For the converse, suppose that there is $N\in {\mathcal P}_r$ with $(M,N)\neq 0$.  The image, $M''$ say, of some non-zero morphism from $M$ to $N$ is finite-dimensional, so ${\rm Ext}^1(\tau^{-1}N,M'')\neq 0$ (we may assume that the slope of $N$ is $>0$, so $\tau^{-1}N$ is defined), and we have an exact sequence $0\rightarrow M' \rightarrow M \rightarrow M'' \rightarrow 0$.  By \ref{pdid1}, ${\rm Ext}^2(\tau^{-1}N,-)=0$ so ${\rm Ext}^1(\tau^{-1}N,M)$ maps onto ${\rm Ext}^1(\tau^{-1}N,M'')$, hence ${\rm Ext}^1({\cal P}_r,M)\neq 0$, as required.

\begin{theorem}\label{indecslope}\marginpar{indecslope} \cite[13.1]{ReiRin} Every indecomposable module over a tubular algebra has a slope.
\end{theorem}

\subsection{Pp formulas}

The required background from the model theory of modules may be found in various references, for instance \cite{PreNBK} (or \cite{PreHoAMod} or \cite{PreBk}).

A {\bf pp formula} $\phi(v)$ with free variable $v$ is a system of homogenous $A$-linear equations with all of the unknowns except $v$ existentially quantified out.  For instance $\exists w \, (rv+sw=0)$, where $r,s\in A$, is a pp formula (for left modules) with free variable $v$ and the general form is $\exists v_2,\dots, v_n \, H(v, v_2,\dots, v_n)^T=0$ where $H$ is a matrix with entries from $A$.  If $\phi =\phi(v)$ is a pp formula and $M$ is any $A$-module then $\phi(M)$ denotes the solution set of $\phi$ in $M$ - a {\bf pp-definable subgroup} of $M$.  We regard as equivalent pp formulas which have the same solution set in every module, equivalently, see \cite[1.2.23]{PreNBK}, in every finitely presented module.  We will write ${\rm pp}_A$ for the set of equivalence classes and we often identify a pp formula with its equivalence class.  More generally if $M$ is a module and $\phi, \psi \in {\rm pp}_A$ then we set $\phi \sim_M \psi$ if $\phi(M)=\psi(M)$ and we write ${\rm pp}(M)$ for the quotient ${\rm pp}_A/\sim_M$; this can be identified with the set, indeed lattice under intersection and sum, of pp-definable subgroups of $M$ and we write $\phi \wedge \psi$ and $\phi + \psi$ for the corresponding operations on (equivalence classes of) pp formulas.  We also write $\psi \leq \phi$ if $\psi(M)\leq \phi(M)$ for every (finitely presented) module $M$ and, if this holds, refer to the {\bf pp-pair} $\phi/\psi$, saying that it is {\bf open} on a module $M$ if $\psi(M)<\phi(M)$, and {\bf closed} on $M$ if $\psi(M)=\phi(M)$.

If $m \in M \in {\rm mod}\mbox{-}A$ then, see \cite[1.2.6]{PreNBK}, there is a pp formula $\phi$ which {\bf generates the pp-type} of $m$ in $M$, meaning first that $m\in \phi(M)$ and $\phi(M)$ is the smallest pp-definable subgroup of $M$ containing $m$ but, further, $\phi$ is minimal in ${\rm pp}_A$ with this property.  If $\phi$ is a pp formula then a {\bf free realisation} of $\phi$ is a pointed module $(M,m)$ (that is, $m\in M$) with $M$ finitely presented, such that $\phi$ generates the pp-type of $m$ in $M$.  We will also use $ m $ to denote the morphism $ A\rightarrow M $ which is defined by taking $ 1 $ to $ m$; $ {\rm coker}(m) $ usually will refer to the module $ M/Am $ rather than the cokernel map.  Note that the natural surjection $ M\rightarrow {\rm coker}(m)$ induces an embedding of the functor $ ({\rm coker}(m), -) $ into $ (M,-) $.

We will use the following results, often without further comment.

\begin{prop}\label{freereal}\marginpar{freereal} (see \cite[1.2.14, 1.2.17]{PreNBK}) Every pp formula has a free realisation.  If $ (M,m) $ is a free realisation of the pp formula $ \phi(v)  $ and if $N$ is any module and $n\in N$ then $ n\in \phi( N)$ iff there is a
morphism $ f:M\rightarrow N $ with $ f(m)=n.$
\end{prop}

So, if $(M,m)$, $(N,n)$ are respectively free realisations of $\phi$, resp.~$\psi$, then $\phi \geq \psi$ iff there is a morphism from $M$ to $N$ taking $m$ to $n$.

\begin{lemma}\label{realpp}\marginpar{realpp} (see \cite[1.2.19]{PreNBK}) If $ (M,m) $ is a free realisation of the pp formula $ \phi (v) $ then for any module $ N$, $ \phi (N)\simeq (M,N)/({\rm coker}(m),N) $ as vector spaces, the isomorphism being induced by $f\in (M,N) \mapsto f(m)$.
\end{lemma}

\begin{lemma}\label{ppmeetsum}\marginpar{ppmeetsum} (see \cite[1.2.27, 1.2.28]{PreNBK}) If $ (M,m)$ is a free realisation of $\phi $ and $(M',m')$ is a free realisation of $\psi$ then $(P,gm=g'm')$ is a free realisation of $\phi \wedge \psi$, where $P$ is the pushout as shown.

$\xymatrix{A \ar[r]^{m} \ar[d]_{m'} & M
\ar[d]^g \\ M' \ar[r]_{g'} & P}.$

The pointed module $(M'\oplus M, (m',m))$ is a free realisation of $\psi+\phi$.
\end{lemma}

A {\bf definable subcategory} ${\mathcal D}$ of $A\mbox{-}{\rm Mod}$ is one closed under direct products, direct limits and pure submodules.  Equivalently, see e.g.~\cite[3.4.7]{PreNBK}, there is a set of pp-pairs such that ${\mathcal D}$ is the collection of all modules on which each of these pp-pairs is closed.  Recall that, to give one of many equivalent definitions, see \cite[\S 2.1]{PreNBK}, a {\bf pure} submodule $A$ of $B$ is one which satisfies $\phi(A)=A\cap \phi(B)$ for every pp formula $\phi$.  Intersections of kernels of covariant Hom and Ext functors are definable on account of the following (\cite[pp.~211-12]{PreInterp} or see \cite[10.2.35, 10.2.36]{PreNBK}).

\begin{theorem}\label{extfpfp}\marginpar{extfpfp}  Let $M\in A\mbox{-}{\rm Mod}$.

\noindent (a) If $M$ is finitely presented then there is a pp-pair $\phi/\psi$ such that the functors $(M,-)$ and $\phi(-)/\psi(-)$ are isomorphic (as objects of the functor category $(A\mbox{-}{\rm Mod}, {\bf Ab})$).

\noindent (b) If $M$ is FP$_2$ - has a projective presentation with the first three terms finitely generated - then the functor ${\rm Ext}^1(M,-)$ on $A\mbox{-}{\rm Mod}$ is isomorphic to one of the form $\phi(-)/\psi(-)$ for some pp-pair $\phi/\psi$.
\end{theorem}

\subsection{Pure-injective modules}

Over a finite-dimensional algebra $A$ a module is pure-injective (that is, injective over pure embeddings) iff it is a direct summand of a direct product of finite-dimensional modules.  Every pure-injective module $N$ decomposes as a direct sum $N_{\rm d}\oplus N_{\rm c}$ where $N_{\rm d}$ is the pure-injective hull of a direct sum of indecomposable pure-injectives and where $N_{\rm c}$ is {\bf superdecomposable}, meaning that it has no indecomposable direct summands.  Over some rings, for instance the ring $k[T]$ of polynomials over a field in one indeterminate, there are no superdecomposable pure-injective modules and for some time it seemed that existence of a superdecomposable pure-injective might be an indication of wildness of the category of finite-dimensional modules.

Indeed Ziegler introduced a notion of width of a modular lattice and had shown, \cite[7.1(1)]{Zie}, that if there is a superdecomposable pure-injective module then the lattice of pp formulas, equivalently the lattice of pointed finitely presented modules, has width undefined (that is, ``$\infty$").  Ziegler also proved, \cite[7.1(2)]{Zie}, the converse when the base ring is countable.  An essentially equivalent notion, the ``breadth" of a modular lattice, which is defined in terms of successively (transfinitely) collapsing intervals which are chains, was used in \cite{PreBk}.  Thus, at least if the ring is countable, existence of a superdecomposable pure-injective is equivalent to the category of finitely presented modules having a certain degree of complexity.

Puninski showed, however, that the modules over any non-domestic string (hence tame) algebra do have this degree of complexity - the width is undefined and so, if the ring is countable (and more generally, see \cite{PunRealSup}), there is a superdecomposable pure-injective.  The results here prove that for another class of tame algebras the width is undefined.  Recently Kasjan and Pastuszak \cite{KasPas} proved the same for strongly simply connected algebras of non-polynomial growth.  In the light of all this and of results such as \cite{PreSchr}, \cite{PrePun}, \cite{Pun}, \cite{PunPreDom} for domestic string algebras, a more reasonable conjecture now is that this dimension, like, conjecturally, Krull-Gabriel dimension, detects the difference between domestic and non-domestic representation type.

The results of Ziegler referred to above also hold in a relative version which applies to any definable category.  The lattice of pp formulas for ${\mathcal D}_r$ is defined at the start of Section \ref{secwide}; the general definition should be obvious from that.

\begin{theorem}\label{spdecwdth}\marginpar{spdecwdth} \cite[7.1]{Zie} If ${\mathcal D}$ is a definable subcategory of $A\mbox{-}{\rm Mod}$ and if there is a superdecomposable pure-injective in ${\mathcal D}$ then the width of the lattice of pp formulas for ${\mathcal D}$ is undefined.  The converse holds if $A$ (or just the lattice of pp formulas for ${\mathcal D}$) is countable.
\end{theorem}

The converse in the uncountable case is open.

At a few places we mention the Ziegler spectrum of $A$.  This is a topological space whose points are the isomorphism classes of indecomposable pure-injective $A$-modules and whose topology is such that the closed sets are in natural bijective correspondence with the definable subcategories of $A\mbox{-}{\rm Mod}$.  In one direction the map takes a definable subcategory to the set of indecomposable pure-injectives in it.  In particular a definable category is generated as such by the module which is the direct sum of each indecomposable pure-injective in it.  For this and other results around the Ziegler spectrum see the original paper \cite{Zie} or, for some perhaps more convenient statements, \cite{PreNBK}, specifically \cite[5.1.4]{PreNBK}.

The Ziegler spectra of tame hereditary algebras are known through the work of a series of authors, culminating in \cite{PreZg} and \cite{RinZie}.  More generally, the Ziegler spectrum of the definable category generated by the modules in a generalised tube was described in \cite{KraGenc}, see also \cite[\S 3.5]{Har} and \cite{RinTubes}, so we understand quite well the modules of rational slope over a tubular algebra.

\section{Modules at irrational cuts}\label{secirrat}\marginpar{secirrat}

We assume throughout this section that $A$ is a tubular algebra, so its category of finite-dimensional modules has the shape described in \ref{canonmodcat}.

We will say that a {\em finite-dimensional} indecomposable module $ X $ is {\bf in} an interval $ I\subseteq {\mathbb Q}^\infty _0$ if its slope lies in this interval; we also extend the terminology to arbitrary finite-dimensional modules to mean that every indecomposable summand is in this interval.

Let $ r $ be a positive irrational. The modules of slope $r$ form, by \ref{extfpfp}, a definable  subcategory $ {\mathcal D}_r $ of $ {\rm Mod}\mbox{-}A $, namely that defined by the following conditions: $ (X,-)=0 $ for every $ X\in {\mathcal Q}_r$; $ {\rm Ext}(Y,-)=0 $ for every $ Y\in {\mathcal P}_r$.  It is immediate from the Compactness Theorem of model theory that this class contains nonzero modules.  Since we will use it elsewhere in this paper, we state this result formally, although we must refer to the background references for some of the terms used in its statement.

\begin{theorem}\label{cpct}\marginpar{cpct} (Compactness Theorem) Let $\Phi$ be a set of sentences (i.e. formulas with no free variables) of a formal first-order language.  Suppose that for every finite subset $\Phi'$ of $\Phi$ there is a structure which satisfies all the sentences in $\Phi'$.  Then there is a structure which satisfies all the sentences in $\Phi$.
\end{theorem}

In our case, the formal language is one suitable for $A$-modules and we have already said in \ref{extfpfp} that each condition of the form $(X,-)=0$ or ${\rm Ext}(X,-)=0$ with $X$ finite-dimensional is equivalent to closure of a pp-pair  $\phi/\psi$ - hence to the sentence $\forall x\, (\phi(x)\rightarrow \psi(x))$ being satisfied.  Let $\Phi$ be the set of all these ${\rm Ext}$ and ${\rm Hom}$ conditions which cut out ${\mathcal D}_r$, together with a sentence expressing the fact that a module is non-zero ($\exists x \, x\neq 0$ will do).  Given any finitely many of these sentences in $\Phi$, there is a module (in $A\mbox{-}{\rm ind}$) which satisfies them so, by \ref{cpct}, there is a module which satisfies all the sentences in $\Phi$ - that is, there is a nonzero module in ${\mathcal D}_r$.  (We remark that, although an alternative is to use compactness of the Ziegler spectrum, that is itself a consequence of \ref{cpct}.)  Being a nonzero definable category, ${\mathcal D}_r$ contains at least one indecomposable pure-injective module (\cite[4.7, 4.10]{Zie}, see \cite[5.1.5]{PreNBK}).

By the result, \ref{indecslope}, of Reiten and Ringel every indecomposable module in $ {\mathcal D}_r $ has slope, necessarily $ r$. Let $ M(r) $ denote any module which generates $ {\mathcal D}_r $ as a definable category, for instance take $ M(r) $ to be the direct sum of all indecomposable pure-injectives of slope $ r$ (see \cite[5.1.6]{PreNBK}).

Given pp formulas $ \phi  $ and $ \phi ' $, set $ \phi \sim _r\phi ' $ if $ \phi (M(r))= \phi '(M(r))$, in which case we say that $ \phi  $ and $ \phi ' $ are {\bf equivalent at} $ r$ (since this implies that they are equivalent on every module in ${\mathcal D}_r$). The map $ \phi \mapsto \phi (M(r))$ induces an isomorphism between the lattice of pp-definable subgroups of $ M(r) $ and the quotient of ${\rm pp}_A$ by this equivalence relation: $ {\rm pp}(M(r)) \simeq {\rm pp}_A/\sim _r$.

\begin{theorem}\label{ppashom}\marginpar{ppashom} Let $ r $ be a positive irrational and let $ \phi (v) $ be pp. Then there is a pp formula $ \phi '  $, a free realisation $ (M',m') $ of $ \phi ' $ and $\epsilon>0$ such that:

$\bullet$ $M'\in  {\rm add}({\mathcal P}_{r-\epsilon })$;

$\bullet$ ${\rm coker}(m') \in  {\rm add}({\mathcal Q}_{r+\epsilon })$;

$\bullet$ $\phi (X)=\phi '(X) $ for all $ X \in A\mbox{-}{\rm mod}$ in $ (r-\epsilon , r+\epsilon )$;

$\bullet$ ${\rm dim}\,\phi (X)= {\rm dim}\,\phi '(X) ={\rm dim}\,(M',X) $ for all $ X$ in $ (r-\epsilon ,r+\epsilon )$

(indeed, evaluation at $m'$ induces a bijection between $(M',X)$ and $\phi(X)$ for $ X$ in $ (r-\epsilon ,r+\epsilon )$).
\end{theorem}

\begin{proof} We strongly use the fact that morphisms in $A\mbox{-}{\rm mod}$ `go from left to right' (see \ref{canonmodcat}).  Let $ (N,n) $ be a free realisation of $ \phi $. Decompose $ N $ as $ M\oplus N_R $ with $ M\in {\rm add}({\mathcal P}_r) $ and $ N_R\in {\rm add}({\mathcal Q}_r) $ and set $ n=m+l$ accordingly.  Also decompose $ M/Am $ as $ C_L\oplus C_R $ with $ C_L\in {\rm add}({\mathcal P}_r) $ and $ C_R\in {\rm add}({\mathcal Q}_r)$. Denote by $ \pi _L $ and $ \pi _R $ the respective compositions of the map $ M\rightarrow M/Am $ with the projections to $C_L$ and $C_R$; set $ K_L={\rm ker}(\pi _L) $ and $ K_R={\rm ker}(\pi _R)$. Notice that both $ K_L $ and $ K_R $ are in $ {\rm add}(P_r) $ and that $ m $ lies in their intersection, indeed, generates that. Since also $ K_L+K_R=M $ we have $ K_L/Am\simeq C_R$.

Let $ \phi ' $ be a pp formula which generates the pp-type of $ m $ in $ K_L$. Choose $ \epsilon >0 $ such that no indecomposable summand of $ N$, $ C_L$, $ C_R$, $ K_R $ or $ K_L $ is in $ (r-\epsilon ,r+\epsilon )$.

Given any $ X $ in $ (r-\epsilon ,r+\epsilon)$, $ (C_R,X)=0 $ so $ (M/Am,X)\simeq (C_L,X)$. Therefore, by \ref{freereal}, \ref{realpp}, and since $(N_L,X)=0$,

${\rm dim}\,\phi (X) ={\rm dim}(M,X) -{\rm dim}(C_L,X)$.

\noindent Similarly, since $ C_R\simeq K_L/Am $ and, therefore, for $ X $ in $ (r-\epsilon , r+\epsilon)$, $ (K_L/Am,X)=0$, we have

${\rm dim}\,\phi '(X)={\rm dim}(K_L,X)$.

\noindent From the exact sequence $ 0\rightarrow K_L\rightarrow  M\rightarrow C_L \rightarrow 0 $ we have the exact sequence

$ 0\rightarrow (C_L,X) \rightarrow (M,X) \rightarrow  (K_L,X)\rightarrow  {\rm Ext}(C_L,X)=0 $

\noindent since $ {\rm dim}\, {\rm Ext}\, (C_L,X) ={\rm dim} (\tau^{-1}X,C_L)=0. $ Therefore:

${\rm dim}\,\phi '(X)={\rm dim}(K_L,X) ={\rm dim}(M,X) -{\rm dim}(C_L,X) ={\rm dim}\, \phi (X)$.

\noindent If we let $\phi_L$, respectively $\phi_R$, denote a pp formula generating the pp-type of $m$ in $M$, resp.~of $l$ in $N_R$, then (by the comments after \ref{freereal}) $ \phi ' \geq \phi_L $ and also $\phi_R(X)=0$ for $X$ in $ (r-\epsilon ,r+\epsilon)$, so $ \phi (X)=\phi '(X) $ for all $ X $ in $ (r-\epsilon ,r+\epsilon)$.  Setting $ (M',m')=(K_L,m) $ completes the proof (the last statement follows from the above lines and \ref{realpp}).
\end{proof}

In fact, the restriction to $X$ being finite-dimensional in the listed properties is not necessary.  We show this (\ref{ppashominf}) next.

\begin{prop}\label{ReiRinLem}\marginpar{ReiRinLem} (part of \cite[Lemma 11]{ReiRin}) For every $r\in {\mathbb R}^+\cup \{\infty\}$, every module satisfying $(M,{\mathcal P}_r)=0$ is generated by ${\mathcal T}_q$ for every rational $q$ with $0<q<r$.
\end{prop}

Using this, we derive the following.

\begin{lemma}\label{limit}\marginpar{limit} Suppose that $M$ is a module of positive slope $r>0$.  Then for every $\epsilon>0$, $M$ is the directed union of its finite-dimensional submodules in $(r-\epsilon,r]$, indeed in $(r-\epsilon,r)$ in the case that $r$ is irrational.
\end{lemma}

\begin{proof} Choose a rational $q$ in $(r-\epsilon, r)$; then $M$ is, by \ref{ReiRinLem}, the directed union of images of morphisms from modules of slope $q$ and, since $({\mathcal Q}_r,M)=0$, the image of any such morphism has no component with slope $>r$ hence is in $[q,r] \subseteq (r-\epsilon,r]$ and, in the case that $r$ is irrational, is in $[q,r) \subseteq (r-\epsilon,r)$.
\end{proof}

Following \cite{LenzTrans}, say that a module $M$ is {\bf supported on} (or has support in) an interval $(a,b)$ if it is a direct limit of finite-dimensional modules with slope in $(a,b)$.  So we have just seen that every $M$ with slope $r$ is supported on $(r-\epsilon,r+\epsilon)$.

\begin{cor}\label{ppashominf}\marginpar{ppashominf} With assumptions and notation as in \ref{ppashom} there is $\epsilon>0$ such that the conclusions hold also for every module $X\in A\mbox{-}{\rm Mod}$ of slope in (or, more generally, which is supported on) $(r-\epsilon, r+\epsilon)$.
\end{cor}

\begin{proof} If $\phi$ is a pp formula then the functor $\phi(-)$ commutes with direct limits (e.g.~\cite[1.2.31]{PreNBK}), as do representable functors $(X,-)$ where $X$ is finitely presented, so this follows directly from the conclusions of \ref{ppashom}.
\end{proof}

Let us also say that $M\in A\mbox{-}{\rm Mod}$ {\bf lies over} the interval $(a,b)$ if $ (M,X)=0 $ for all $X\in A\mbox{-}{\rm ind}$ of slope $\leq a$ and $(X,M)=0$ for all $X\in A\mbox{-}{\rm ind}$ of slope $\geq b$.  For finite-dimensional modules $M$ these conditions, of being supported on and of lying over an open interval, are equivalent and, by \ref{canonmodcat}, are what we already have referred to as being ``in" $(a,b)$.  Write ${\mathcal D}_{(a,b)}$ for the category of modules supported in $(a,b)$ and ${\mathcal D}^+_{(a,b)}$ for the category of modules which lie over $(a,b)$; by \cite[2.1]{LenzTrans} and \ref{extfpfp} respectively these are definable subcategories of $A\mbox{-}{\rm Mod}$.

\begin{lemma}\label{ddplus}\marginpar{ddplus} Given an open interval $(a,b) \subseteq {\mathbb R}^+$, every module in ${\mathcal D}_{(a,b)}$ is a union of its finite-dimensional modules in $(a,b)$ and we have ${\mathcal D}_{(a,b)} \subseteq {\mathcal D}^+_{(a,b)}$.
\end{lemma}
\begin{proof} The first statement is by (the proof of) \ref{limit}.  Write $M\in {\mathcal D}_{(a,b)}$ as a direct limit $\varinjlim_i M_i$ of finite-dimensional modules with slope in $(a,b)$.  For each $i$ we have $(M_i,X)=0$ for each finite-dimensional $X$ with slope $\leq a$, hence $(M,X)=0$, by definition of direct limit.  If $X\in A\mbox{-}{\rm ind}$ has slope $\geq b$ then $(X,M_i)=0$ for all $i$ so, since $X$ is finitely presented, $(X,M)=0$ as required.
\end{proof}

Since every definable subcategory is determined by the indecomposable pure-injectives in it, we can describe the difference between ${\mathcal D}_{(a,b)} $ and $ {\mathcal D}^+_{(a,b)}$ in terms of these (the infinite-dimensional ones, since these categories contain the same finite-dimensional modules).  Note that the indecomposable pure-injectives in ${\mathcal D}_{(a,b)}$ are (in consequence of \ref{indecslope}): the finite-dimensional ones with (rational) slope in $(a,b)$; the adic, generic and Pr\"{u}fer modules of rational slope in $(a,b)$; those with irrational slope in $(a,b)$.  Since every finite-dimensional indecomposable is an open point of the Ziegler spectrum (\cite[13.1]{PreBk}, see \cite[5.3.33]{PreNBK}), the definable subcategory generated by the infinite-dimensional indecomposable pure-injectives in ${\mathcal D}_{(a,b)} $ (or $ {\mathcal D}^+_{(a,b)}$) will contain no finite-dimensional modules so could be seen as the ``infinite part" of the category of modules supported on (or lying over) the interval $(a,b)$.

\begin{lemma}\label{ddplusdiff}\marginpar{ddplusdiff} Take $0<a<b \in {\mathbb R}$.

\noindent (a) If $a$ is irrational then the infinite-dimensional indecomposable pure-injective modules in ${\mathcal D}_{(a,b)}^+ \setminus {\mathcal D}_{(a,b)}$ are exactly those of slope $a$.

\noindent (b) If $a$ is rational then the infinite-dimensional indecomposable pure-injective modules in ${\mathcal D}_{(a,b)}^+ \setminus {\mathcal D}_{(a,b)}$ are exactly the Pr\"{u}fer modules of slope $a$ and the indecomposable generic module of slope $a$.
\end{lemma}

\begin{proof} By \ref{limit} any indecomposable pure-injective in ${\mathcal D}_{(a,b)}^+ \setminus {\mathcal D}_{(a,b)}$ must have slope $a$.  In the case that $a$ is irrational and $M$ has slope $a$ then $M$ is in ${\mathcal D}_{(a,b)}^+$ but any finite-dimensional module with a non-zero morphism to $M$ must have slope $<a$ so $M$ is not in ${\mathcal D}_{(a,b)}$.

For (b), if $a$ is rational and $M$ is an indecomposable pure-injective with slope $a$, then it is is adic, Pr\"{u}fer or generic.  By \cite[Thm.~4]{ReiRin}, the Pr\"{u}fer and generic modules of slope $a$ are exactly the indecomposable pure-injectives with no nonzero morphism to a finite-dimensional module of slope $a$ and hence which satisfy the defining conditions for ${\mathcal D}^+_{(a,b)}$.
\end{proof}

\begin{cor}\label{dijoint}\marginpar{disjoint} If $r\neq s$ are positive real numbers then ${\mathcal D}_r \cap {\mathcal D}_s =0$.  If $(a,b)$ and $(c,d)$ are disjoint open intervals with $b\leq c$ then ${\mathcal D}_{(a,b)} \cap {\mathcal D}_{(c,d)} =0$.
\end{cor}
\begin{proof} The first statement follows from uniqueness of slopes since every nonzero definable category contains an indecomposable pure-injective.  The second, a special case of \cite[2.7]{LenzTrans}, follows directly from \ref{ddplus}.
\end{proof}

The uniqueness of slope, \ref{indecslope}, means that no infinite-dimensional module can ``disappear from view over an interval" as far as the finite-dimensional modules are concerned.  That is, there is no module $M$ such that $(M, {\mathcal P}_s)=0$ and $({\mathcal Q}_r,M)=0$ with $r<s$.  For, by \cite[6.9]{Zie}, $M$ is elementarily equivalent to a direct sum of indecomposable pure-injectives, each of which, since it would have to satisfy these (equivalent to) pp conditions, would have to have slope both $<r$ and $>s$, contradicting \ref{indecslope}.

\begin{cor}\label{vdotdim}\marginpar{vdotdim} Let $ \phi /\psi  $ be a pp-pair and let $ r $ be a positive irrational. Then there is $ \epsilon >0 $ and a vector $ v\in K_0(A) $ such that $ {\rm dim}((\phi /\psi )(X))=v\cdot \underline{\rm dim}(X) $ for all $ X $ in $ (r-\epsilon ,r+\epsilon )$.
\end{cor}

\begin{proof} Let $M'$ and $\epsilon$ be as in \ref{ppashom} for $\phi$. Since the slope of $M'$ is $<\infty$, by \ref{pdid1} and the fact that modules in ${\mathcal P}_0 \cup {\mathcal T}_0$ have projective dimension $\leq 1$ (\cite[3.1(5)]{RinTame}), there is an exact sequence $ 0\rightarrow P' \rightarrow P \rightarrow  M'\rightarrow 0 $ with $ P $ and $ P' $ projective. This induces the exact sequence $ 0\rightarrow (M',X)\rightarrow  (P,X)\rightarrow (P',X) \rightarrow {\rm Ext}(M',X)=0 $ for $X$ in $(r-\epsilon, r+\epsilon)$ and consequently, using \ref{ppashom},

${\rm dim}(M',X) ={\rm dim}(P,X) -{\rm dim}(P',X)$.

Let $ P = P_1^{c_1} \oplus \dots \oplus P_n^{c_n} $ and $ P' = P_1^{d_1} \oplus \dots \oplus P_n^{d_n} $ and set $ v_1=(c_1-d_1, \dots, c_n-d_n)$. Then, for all $ X $ in $ (r-\epsilon ,r+\epsilon )$:

${\rm dim}\,\phi (X)={\rm dim}(M',X) =v_1\cdot \underline{\rm dim}(X)$.

Similarly, there is $\delta >0 $ and $ v_2\in K_0(A) $ such that for all $ X$ in $(r-\delta, r+\delta)$:

${\rm dim}\,\psi (X)=v_2\cdot \underline{\rm dim}(X)$.

Taking $ v=v_1-v_2 $ and relabelling $ {\rm min}(\epsilon ,\delta ) $ as $ \epsilon $ completes the proof.
\end{proof}

Say that a pp-pair $ \phi /\psi  $ is {\bf closed near the left of} $ r $ if there is $ \epsilon >0 $ such that $ \phi /\psi  $ is closed on every indecomposable $ X $ in $ (r-\epsilon ,r)$; otherwise say that $ \phi /\psi  $ is {\bf open near the left of} $ r$.  The latter says only that $ \phi /\psi  $ is open on ``cofinally many" modules near to, and to the left of, $ r $ but it will be proved (\ref{homogopen} then \ref{allopenlh}) that, in this case, $ \phi /\psi  $ is open on every module in some interval $ (r-\epsilon , r)$. Similarly say that $\phi /\psi  $ is {\bf closed near the right of} $ r $ if there is $ \epsilon >0 $ such that $ \phi /\psi  $ is closed on every indecomposable $ X $ in $ (r,r+\epsilon)$; otherwise say that $ \phi /\psi  $ is {\bf open near the right of} $ r$.

If $E$ is a quasisimple module (that is, a module at the mouth of a tube); then we will denote by $E[k]$ the module in the same tube which has quasisimple length $k$ and quasisimple socle $E$ (for the structure of modules in tubes we refer to the background references).

\begin{cor}\label{homogopen}\marginpar{homogopen} Let $ \phi /\psi  $ be a pp-pair and let $ r $ be a positive irrational.

If $ \phi /\psi  $ is open near the left of $ r$ then there is $ \epsilon >0 $ such that $ \phi /\psi  $ is open on every module in $ (r-\epsilon ,r) $ which lies in a homogeneous tube.

Similarly, if $ \phi /\psi  $ is open near the right of $ r$ then there is $ \epsilon >0 $ such that $ \phi /\psi  $ is open on every module in $ (r,r+\epsilon ) $ which lies in a homogeneous tube.
\end{cor}

\begin{proof} Apply \ref{ppashom} to obtain pp formulas $ \phi ' $ and $ \psi ' $, with free realisations $ (M',m') $ and $ (N',n') $, and $ \epsilon _1, \epsilon _2 $ satisfying the conclusions of that result. Set $ \epsilon ={\rm min}(\epsilon _1,\epsilon _2)$.

Suppose that there is $ \gamma \in (r-\epsilon ,r) \cap {\mathbb Q} $ and a module $ E[k] $ in a homogeneous tube $ {\mathcal T}(\rho ) $ in $ {\mathcal T}_\gamma  $ such that $ \phi /\psi  $ is closed on $ E[k]$. We shall prove that $ \phi /\psi  $ must be closed near the left of $ r$.

We have $ \phi '(E[k])=\phi (E[k]) =\psi (E[k])=\psi '(E[k]) $ and so, by \ref{ppashom}, $ {\rm dim}(M',E[k])={\rm dim}(N',E[k])$.  By considering almost split sequences in ${\mathcal T}(\rho)$ and induction, (or using \ref{vdotdim}) it is easy to check that for all positive integers $ m$, $ {\rm dim}(M',E[m])={\rm dim}(N',E[m])$. That is $ \phi '/\psi ' $ is closed on every module in $ {\mathcal T}(\rho )$ and hence on every module in ${\rm add}({\mathcal T}(\rho ))$.

Now, given $ X $ (with slope) in $ (\gamma ,r) $ and any $ x\in \phi (X)=\phi '(X)$, there is, by \ref{freereal}, $ f\in (M',X) $ such that $ f(m')=x$. By \ref{canonmodcat}, $ f $ factors through a module $ Y\in {\rm add}({\mathcal T}(\rho ))$.

$\xymatrix{M' \ar[rr]^f \ar@{.>}[dr]_{\exists g} &  & X\\ &  Y \ar@{.>}[ur]^{\exists h}}$

Since $ \phi '/\psi ' $ is closed on $ Y $ it follows that $ g(m')\in \psi '(Y) $ and so, since solution sets of pp formulas are preserved by homomorphisms, $ x\in \psi '(X)=\psi (X)$. Thus $ \phi /\psi  $ is closed on every module in $ (\gamma ,r) $ - as required.

For the second statement, by the argument above, if $\phi / \psi$ is open on some module $X$ of slope $\delta \in (r, r+\epsilon)$ then it is open on every module $E[k]$ in a homogeneous tube with slope in $ (r,\delta)$.
\end{proof}

\begin{cor}\label{cofinrtl}\marginpar{cofinrtl} Let $ \phi /\psi  $ be a pp-pair and let $ r $ be a positive irrational.  Then the following are equivalent:

\noindent (i) $ \phi /\psi  $ is open near the left of $ r$;

\noindent (ii) $ \phi /\psi  $ is open near the right of $r$;

\noindent (iii) $ \phi /\psi  $ is open at $r$, that is, in some module of slope $r$.
\end{cor}

\begin{proof}  If we have (i) then we can use a compactness argument very similar to that used earlier to get (iii).  Namely, given finitely many of the (Hom and Ext) conditions cutting out the subcategory ${\mathcal D}_r$, there is $\epsilon>0$ such that every indecomposable finite-dimensional module with slope in $(r-\epsilon, r)$ satisfies them.  By assumption there is such a module on which $\phi/\psi$ is open and that can be expressed by the sentence $\exists x \, (\phi(x) \wedge \neg \psi(x))$ (where $\wedge$ is read as ``and" and $\neg$ is read as ``not").  Thus the conditions cutting out ${\mathcal D}_r$ are finitely consistent with the condition that $\phi/\psi$ is open so, by the Compactness Theorem, there is a module which satisfies all these conditions, hence is of slope $r$, as required.  This argument also shows (ii)$\Rightarrow$(iii).

For the converse, suppose we have (iii), say  $M$ is a module of slope $r$ on which $\phi/\psi$ is open, say $a\in \phi(M)\setminus \psi(M)$.  Choose a finite-dimensional submodule $Y$ of $M$ containing $a$ and such that $a\in \phi(Y)$:  if $\phi(x)$ is the formula $\exists x_2,\dots,x_n \, H(x,x_2,\dots, x_n)^T=0$ where $H$ is a matrix with entries in $A$ then choose elements $a_2, \dots, a_n \in M$ such that $H(a,a_2,\dots,a_n)^T=0$ and let $Y$ be the submodule of $M$ generated by $a, a_2,\dots, a_n$.  Given $\epsilon>0$ there is, by \ref{limit}, a finite-dimensional submodule $X$ of $M$ which contains $Y$ and is in $(r-\epsilon, r)$.  Then certainly $a\in \phi(X)$ and, since $a\notin \psi(M)$ it must be that $a\notin \psi(X)$.  Therefore $\phi/\psi$ is open on $X$ and hence is open on some indecomposable summand of $X$, and we have proved (i).

Finally, assume (iii) and, in order to prove (ii), continue with notations and assumptions as in the previous paragraph.  By \ref{limit} there is a submodule $Z$ of $M$ into which $X$ embeds, such that every indecomposable summand of $Z$ has slope greater than the maximum slope of any indecomposable summand of $X$.  By the factorisation property the embedding $X\rightarrow Z$ factors through a module $X'$ whose indecomposable summands all lie in some homogeneous tube of slope between that of any direct summand of $X$ and $r$.  Note that if $b$ denotes the image in $X'$ of $a$ regarded as an element of $X$, then $b'\in \phi(X')\setminus \psi(X')$ (since $a\in \phi(M)\setminus \psi(M)$).  If the vector $v$ is chosen for $\phi/\psi$ as in \ref{vdotdim} then we have $v.\underline{\rm dim}(X')>0$.  By definition of slope, the dimension vector of $X'$ has the form $c(h_0+\gamma h_\infty)$ where $\gamma$ is the slope of $X'$ and $c$ is a positive rational.  Therefore $v.(h_0+\gamma h_\infty)>0$.  Repeating this whole argument (but noting that $v$ can be taken to be fixed), we produce an increasing sequence of rationals (the various $\gamma$) with limit $r$, each satisfying $v.h_0+\gamma v.h_\infty>0$.  Since $r$ is irrational it cannot be that $v.h_0+r v.h_\infty=0$, so $v.h_0+r v.h_\infty>0$ and hence there is a rational $\gamma'>r$, which we may take in the interval $(r,r+\epsilon'')$ for any $\epsilon''>0$, with $v.h_0+\gamma' v.h_\infty>0$.  By \ref{vdotdim}, $\phi/\psi$ is open on the homogeneous modules of slope $\gamma'$.  Thus, $\phi/\psi$ is open near the right of $r$.
\end{proof}

\section{Extending to non-homogeneous tubes}\label{sechomog}\marginpar{sechomog}

In this section it will be shown that if $ r $ is a positive irrational and if $ \phi /\psi  $ is a pp-pair which is open near the left of $ r $ then there is $ \epsilon >0 $ such that $ \phi /\psi  $ is open on every module in $ (r-\epsilon ,r)$.  In view of \ref{homogopen} it is the modules in non-homogeneous tubes which have yet to be dealt with.

The main fact that we need is \ref{ineqomega}; we will also need, for \ref{dimleap}, another dimension estimate, \ref{nravg}, which says that in a non-homogenous tube the dimensions of modules are not far from the average dimension of modules in that tube.  Both these results will be proved for certain tubular algebras and their consequences will, in Section \ref{secmove}, be transferred to the general case using tilting functors.

The particular algebras are those considered in Section 5.6 of \cite{RinTame}, namely the following.

The algebra $C(4,\lambda)$, where $\lambda\in K\setminus\{ 0,1\}$, is the path algebra of the quiver
$$
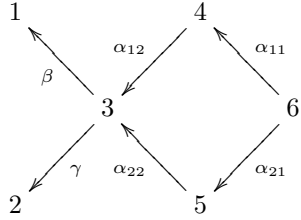
\xymatrix{1 & & 4 \ar[dl]_{\alpha_{12}} \\
& 3 \ar[ul]^{\beta} \ar[dl]^{\gamma} & & 6 \ar[ul]_{\alpha_{11}} \ar[dl]^{\alpha_{21}} \\
2 & & 5 \ar[ul]^{\alpha_{22}} }$$
with relations $\beta(\alpha_{12}\alpha_{11} - \alpha_{22}\alpha_{21})=0$ and $\gamma(\alpha_{12}\alpha_{11} - \lambda \alpha_{22}\alpha_{21}) =0$.

The quadratic form for $C(4,\lambda)$ is
$$\chi(x_1,\dots,x_6) =$$ $$= \dfrac{1}{2}(x_1-x_2)^2 + \big(x_3-\dfrac{1}{2}(x_1+x_2+ x_3+x_4)\big)^2 + \big(x_6+ \dfrac{1}{2}(x_1+x_2 -x_4-x_5)\big)^2 +\dfrac{1}{2}(x_4-x_5)^2.$$  Also $h_0=(1,1,2,1,1,0)$, $h_\infty= (0,0,1,1,1,1)$ and $\langle h_0,h_\infty \rangle =2$.  The slope of a module with dimension vector $(x_1,\dots,x_6)$ is $\dfrac{x_4+x_5-x_1-x_2}{x_3-x_6}$.

The algebra $ C(6)$ is the path algebra of the quiver
$$
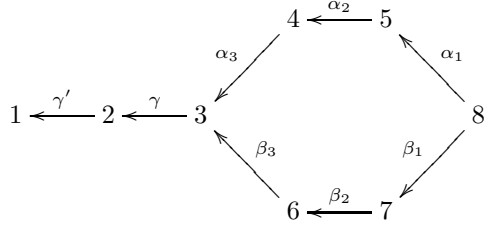
\xymatrix{& & & 4 \ar[dl]_{\alpha_3} & 5 \ar[l]_{\alpha_2} \\
1 & 2 \ar[l]_{\gamma'} & 3 \ar[l]_{\gamma} & & & 8 \ar[ul]_{\alpha_1} \ar[dl]_{\beta_1} \\
& & & 6 \ar[ul]_{\beta_3} & 7 \ar[l]_{\beta_2}}$$
with relations $\gamma( \alpha_3\alpha_2\alpha_1 -\beta_3\beta_2\beta_1)=0$.

For the algebras $C(7)$ and $C(8)$ and for the further details about these algebras we refer to \cite[\S 4.2]{Har} and \cite[\S 5.6]{RinTame}.

{\bf From now on in this section}, we deal only with the algebras $C$ listed above.

\begin{lemma}\label{radchiC}\marginpar{radchiC} If $C$ is one of the algebras $ C(4,\lambda )$, $ C(6)$, $ C(7)$ or $ C(8)$ then ${\rm rad}(\chi_C) =\{ ah_0 +bh_\infty: a,b\in {\mathbb Z}\}$.
\end{lemma}
\begin{proof} We know from the comments immediately after \ref{pdid1} that every element $ (x_1,\dots, x_n) $ of $ {\rm rad} (\chi) $ can be written in the form $ q_1h_0+q_2h_\infty  $ with $ q_1,q_2 $ rational. Notice for these particular algebras that the $ (n-1)$-th and $ n$-th coordinates of $ h_0 $ are 1 and 0 respectively, and that these coordinates for $ h_\infty  $ are both 1.  Projecting onto these coordinates we get $ q_1+q_2=x_{n-1} $ and $ q_2=x_n $ from which we see that $ q_1 $ and $q_2 $ are integers.
\end{proof}

\begin{lemma}\label{hlinear}\marginpar{hlinear} For all $ x\in {\rm K}_0(C) $ we have $ \chi_C(x\pm h_0) =\chi_C(x) =\chi_C(x\pm h_\infty )$.
\end{lemma}
\begin{proof} For each of these algebras $C$ the quadratic form $\chi_C$ is a sum of squares of linear terms $t$ each of which, when evaluated at $h_0$ or $h_\infty$, must therefore be $0$.  It follows that for each such term $t$ we have $t(x\pm h_0)=t(x)$ and so $\chi_C (x\pm h_0)=\chi_C(x)$ and similarly for $\chi_C(x\pm h_\infty)$.
\end{proof}

Define $\Omega =\{ x\in {\rm K}_0(C): \chi_C(x)=1 \mbox{ and } x_{n-1}=x_n =0\}$.

\begin{lemma}\label{bdomega}\marginpar{bdomega} Given one of the above algebras $ C$, there is a bound $ b $ such that if $ x\in \Omega $ then $ |x_i|\leq b$ for all $ i$.  In particular $\Omega$ is finite.
\end{lemma}
\begin{proof} This can be checked for each type of algebra $ C$; we give the argument only for the case $ C=C(4,\lambda )$. In that case we have: $$\chi_C(x_1,\dots, x_6) = \dfrac{1}{2} (x_1-x_2)^2 + (x_3- \dfrac{1}{2} (x_1+x_2+x_4+x_5))^2 + \dfrac{1}{2} (x_4-x_5)^2 + (x_6 + \dfrac{1}{2} (x_1+x_2 -x_4-x_5))^2.$$
So, if $ \chi _C(x)=1$ and $ x_5=x_6=0 $, then: $$\dfrac{1}{2} (x_1-x_2)^2 \leq 1$$ $$(x_3- \dfrac{1}{2} (x_1+x_2+x_4))^2 \leq 1$$ $$\dfrac{1}{2} x_4^2\leq 1$$ $$ \dfrac{1}{2} (x_1+x_2 -x_4)^2\leq 1.$$ From these equations one quickly obtains a bound $b$.
\end{proof}

\begin{lemma}\label{bijomega}\marginpar{bijomega} Every $ x\in {\rm K}_0(C) $ with $\chi_C(x)=1$ may be written in the form $ ah_0+bh_\infty +y $ with $ a,b\in {\mathbb Z} $ and $ y\in \Omega$.
\end{lemma}
\begin{proof} Take any $ x\in {\rm K}_0(C) $ with $ \chi (x)=1 $ and let $ y=x-(x_{n-1}-x_n)h_0 -x_nh_\infty $. By \ref{hlinear}, $ \chi (y)=1 $ so clearly $ y\in \Omega$.
\end{proof}

The next two lemmas can be seen most directly by considering slopes of lines joining the origin to points $(a,b)$ in the plane.

\begin{lemma}\label{ineq+}\marginpar{ineq+} Given $ r_1, r_2 \in {\mathbb R} $ such that $ 0<r_1<r_2 $ and any $ \gamma _1,\gamma _2\in {\mathbb Q}$, there are only finitely many pairs $ (a,b)\in {\mathbb N}^2 $ such that $$ \dfrac{b}{a} \leq  r_1<r_2 \leq \dfrac{b+\gamma _1}{a+\gamma _2}.$$
\end{lemma}

\begin{lemma}\label{ineq-}\marginpar{ineq-} Given $ r_1, r_2 \in {\mathbb R} $ such that $ 0<r_1<r_2 $ and any $ \gamma _1,\gamma _2\in {\mathbb Q}$, there are only finitely many pairs $ (a,b)\in {\mathbb N}^2 $ such that $$ 0< \dfrac{b+\gamma _1}{a+\gamma _2} \leq  r_1< r_2 \leq  \dfrac{b}{a}.$$
\end{lemma}

Recall that $\iota(x)$ denotes the index=slope of a dimension vector $x$.

\begin{cor}\label{ineqomega}\marginpar{ineqomega} Let $C$ be one of the listed algebras.  Take any positive real $ r>0 $ and any $ \epsilon$ with $ 0< \epsilon <r$. Then there exists $ \delta \in (0,\epsilon ) $ such that, for all $ a,b\in {\mathbb N} $ and any $ y\in \Omega$, $$\iota(ah_0+bh_\infty + y )\in  (r-\delta ,r+\delta ) \mbox{ implies } \iota(ah_0+bh_\infty )\in (r-\epsilon ,r+\epsilon ).$$
\end{cor}

\begin{proof} We have $\iota(ah_0+bh_\infty) =b/a$ and  $$\iota(ah_0+bh_\infty + y ) = -\dfrac{\langle h_0,ah_0+bh_\infty +y\rangle }{\langle h_\infty ,ah_0+bh_\infty +y\rangle } = - \dfrac{\langle h_0,bh_\infty \rangle +\langle h_0,y\rangle }{\langle h_\infty ,ah_0\rangle +\langle h_\infty ,y\rangle }  = \dfrac{b+(\langle h_0,y\rangle /\langle h_0,h_\infty \rangle )}{a-(\langle h_\infty ,y\rangle /\langle h_0,h_\infty \rangle )}.$$

Set $ \gamma _1=\langle h_0,y\rangle /\langle h_0,h_\infty \rangle  $ and $ \gamma _2=-\langle h_\infty ,y\rangle /\langle h_0,h_\infty \rangle  $ and pick any $ \epsilon '\in (0,\epsilon )$. Then, by \ref{ineq-} and \ref{bdomega}, there are only finitely many $ (a,b)\in {\mathbb N}^2 $ such that for some $ y\in \Omega $ $$r-\epsilon' < \dfrac{b+\gamma _1}{a+\gamma _2} <  r+\epsilon ' <r+\epsilon  \leq \dfrac{b}{a}.$$ Similarly, by \ref{ineq+}, there are only finitely many $ (a,b)\in {\mathbb N}^2 $ and $ y\in \Omega $ such that: $$\dfrac{b}{a} \leq r-\epsilon <r-\epsilon ' \leq  \dfrac{b+\gamma _1}{a+\gamma _2}.$$  Thus there are only finitely many $(a,b)\in {\mathbb N}^2$ such that $\dfrac{b+\gamma_1}{a+\gamma_2} \in (r-\epsilon', r+\epsilon')$ but $\dfrac{b}{a} \notin (r-\epsilon , r+\epsilon)$ for some $y\in \Omega$.  Consequently we can pick $ \delta \in (0,\epsilon ') $ such that, for all $ (a,b)\in {\mathbb N} $ and $ y\in \Omega$, $$\iota(ah_0+bh_\infty + y )\in  (r-\delta ,r+\delta ) \mbox{ implies } \iota(ah_0+bh_\infty )\in (r-\epsilon ,r+\epsilon ).$$
\end{proof}

\begin{cor}\label{allopenlh}\marginpar{allopenlh} Let $C$ be one of the listed algebras.  Let $ \phi /\psi  $ be any pp-pair and let $ r $ be a positive irrational. If $\phi/\psi$ is open at $r$ then there exists $ \epsilon >0 $ such that $ \phi /\psi  $ is open on every module in $ (r-\epsilon ,r+\epsilon)$.
\end{cor}

\begin{proof}  By \ref{homogopen} there is $ \epsilon '>0 $ such that $ \phi /\psi  $ is open on every indecomposable module in $ (r-\epsilon ',r+\epsilon') $ which is in a homogeneous tube.  By \ref{vdotdim} there is $ \epsilon >0 $ and $ v\in K_0(C) $ such that $ {\rm dim}(\phi /\psi) (X)=v\cdot \underline{\rm dim}(X) $ for all $ X $ in $ (r-\epsilon ,r+\epsilon)$; we may assume that $ \epsilon \leq \epsilon ' $ and that $ \epsilon \in {\mathbb Q}$.

We claim that $ v\cdot h_0+\gamma v\cdot h_\infty >0 $ for all $ \gamma \in (r-\epsilon ,r+\epsilon ) \cap {\mathbb Q}$; to see this, take positive $ c\in {\mathbb N} $ large enough such that $ c\gamma \in {\mathbb N}$ and such that (from the definition of slope) there is an indecomposable module $ X $ in a homogeneous tube with $ \underline{\rm dim}(X)=ch_0+c\gamma h_\infty $. Then $ X $, having slope $\gamma$, is in $ (r-\epsilon ,r+\epsilon ) $ and so $ \phi /\psi  $ is open on $ X$, hence $ c(v\cdot h_0+\gamma v\cdot h_\infty )>0$, so $v\cdot h_0+\gamma v\cdot h_\infty >0$ as required.

Now, let $ s={\rm min}\{ v\cdot h_0+(r-\epsilon)v\cdot h_\infty , v\cdot h_0+(r+\epsilon )v\cdot h_\infty \}$.  Notice that $ s\in {\mathbb R}\setminus {\mathbb Q} $, since $ r\pm \epsilon \in {\mathbb R}\setminus {\mathbb Q}$ (the case $v\cdot h_\infty=0$ would also give our desired conclusion, $s>0$), and that $ s={\rm inf}\{ v\cdot h_0 + \gamma v\cdot h_\infty : \gamma \in (r-\epsilon ,r+\epsilon )\}$. Thus $ s>0$.

By \ref{ineqomega}, there exists $ \delta \in (0,\epsilon ) $ such that, for all $ a,b\in {\mathbb N} $ and $ y\in \Omega$, $$\iota(ah_0+bh_\infty +y)\in (r-\delta ,r+\delta ) \mbox{ implies } \iota(ah_0+bh_\infty )\in (r-\epsilon ,r+\epsilon ).$$

Now, pick any $ a'\in {\mathbb N} $ such that $ a'>-(v\cdot y)/s $ for all $ y\in \Omega$.  By choice of $\delta$, for each value of $a$ there are only finitely many values of $b$ such that $\iota(ah_0+bh_\infty + y )$ is in $(r-\delta, r+\delta)$ for some $y\in \Omega$; so we can pick $ \delta '>0 $ with $\delta'<\delta$ small enough so that $ \iota(ah_0+bh_\infty +y)\notin (r-\delta ',r) $ for all $ a\leq a'$, $ b\in {\mathbb N} $ and $ y\in \Omega$.  We claim that $ \phi /\psi  $ is open on every $ X $ in $ (r-\delta ',r)$. Indeed, given any such $ X$, let $ a,b\in {\mathbb N} $ and $ y\in \Omega $ be such that $ \underline{\rm dim}(X)=ah_0+bh_\infty +y$. Then $ a> a' $ (by our choice of $ \delta '$) and $ b/a\in (r-\epsilon ,r+\epsilon ) $ (by our choice of $ \delta $) and so: $${\rm dim}\,\phi (X) -{\rm dim}\,\psi (X) = v\cdot (ah_0+bh_\infty +y) = av\cdot (h_0+(b/a)h_\infty ) + v\cdot y \geq a's + v\cdot y >0.$$ So $ \phi /\psi  $ is open on $ X$,  as required. Relabelling $ \delta ' $ as $ \epsilon$  completes the proof.
\end{proof}

\begin{cor}\label{allopenl1}\marginpar{allopenl1} Let $ C $ be of the form $ C(4,\lambda )$, $ C(6)$, $ C(7)$ or $ C(8)$. Let $ \phi /\psi  $ be any pp-pair and let $ r $ be a positive irrational. Then the following are equivalent:

\noindent (i) $\phi /\psi $ is open at $ r$;

\noindent (ii) there exists $ \epsilon >0 $ such that $ \phi /\psi  $ is open on every module in $ (r-\epsilon ,r+\epsilon)$.
\end{cor}

\begin{proof} This now follows from \ref{cofinrtl} and \ref{allopenlh}.
\end{proof}

This will be extended to all tubular algebras at \ref{allopenrtl}.

\section{Dimension estimates for $ C$-modules in inhomogeneous tubes} \label{secinhomog}\marginpar{secinhomog}

Throughout this section the algebra $ C $ will continue to be of one of the types $ C(4,\lambda )$, $ C(6)$, $ C(7)$, $ C(8)$.

Define $ \mu :{\rm K}_0(A) \rightarrow {\mathbb N} $ to be the linear map such that $ \mu (\underline{\rm dim}(M)) ={\rm dim}(M)$.  Define an ordering on the set ${\rm rad}^+(\chi) =\{ ah_0+bh_\infty: a,b>0, (a,b)\neq (0,0)\}$ by setting $x<y$ iff $\iota(x)<\iota(y)$ or ($\iota(x)=\iota(y)$ and $\mu(x)<\mu(y)$).  Note that this is a total order.

The next statements can easily be checked.

\begin{lemma}\label{iotaest}\marginpar{iotaest} Take any $ x,y\in {\rm rad}^+(\chi_C ) $ such that $ \iota(x)<\iota(y)$. Then: $$(a) \mbox{ } \iota(x)<\iota(x+y) <\iota(y)$$ $$(b) \mbox{ } {\rm lim}_{n\rightarrow \infty } \, \iota(x+ny) =\iota(y).$$
\end{lemma}

\begin{prop}\label{jumpnrr}\marginpar{jumpnrr} Let $ C $ be one of the specified algebras. Given any positive irrational $ r$, any $ k\in {\mathbb N} $ and any $ \epsilon >0$, there exists $ x\in {\rm rad}^+(\chi _C) $ such that $ r-\epsilon <\iota(x)<r $ and such that, for all $ y\in {\rm rad}^+(\chi _C)$, $$\iota(x)<\iota(y)<r \Rightarrow \mu (y)>\mu (x)+k.$$
\end{prop}
\begin{proof} First of all, given any $ k'\geq 1$, consider the set: $$\{ x\in {\rm rad}^+(\chi ): \iota(x)<r \mbox{ and } \mu (x)\leq k'\}.$$ There exists $ k'\geq k $ such that this set is nonempty; fix such a $k'$. Since the set is finite we can choose $ x_0=a_0h_0+b_0h_\infty  $ in this set which is maximal in the order on $ {\rm rad}^+(\chi )$. Then for all $ y\in {\rm rad}^+(\chi )$: $$ \iota(x_0)<\iota(y)<r \Rightarrow \mu (y)>\mu (x_0).$$

Suppose, for a contradiction, that for all $ x\in {\rm rad}^+(\chi ) $ with $ r-\epsilon <\iota(x)<r $ there exists $ y\in {\rm rad}^+(\chi ) $ with $ \iota(x)<\iota(y)<r $ and $ \mu (y)\leq \mu (x)+k$. Then we can recursively define non-empty finite sets $ S_1, S_2, \dots $ and elements $ x_i=a_ih_0+b_ih_\infty  \in S_i $ by: $$S_{i+1}=\{ y\in {\rm rad}^+(\chi ): \iota (x_i)<\iota (y)< r \mbox{ and } \mu (y)\leq \mu (x_i)+k\}$$ $$x_{i+1}={\rm max}(S_{i+1}).$$ Define $ c_i=a_i-a_{i-1} $ and $ d_i=b_i-b_{i-1} $ for all $ i\geq 1$. So $ x_i-x_{i-1}=c_ih_0+d_ih_\infty $. Notice that for all $ i$:

$\bullet$ $ 0 < \mu (c_ih_0+d_ih_\infty )\leq k $ (since, by induction, $\mu(y)>\mu(x_i)$ for all $y\in S_{i+1}$):

$\bullet$ $ c_i $ and $ d_i $ can't both be negative (since $ 0\leq \mu (c_ih_0+d_ih_\infty )$);

$\bullet$ $ d_i\geq 0 $ - suppose for a contradiction that $ d_i<0$; then $ c_i\geq 0 $ (by above) and so $$\iota(x_i) =b_i/a_i =(d_i+b_{i-1})/(c_i+a_{i-1}) \leq b_{i-1}/a_{i-1} =\iota (x_{i-1})$$ - contradicting the definition of $ S_i$;

$\bullet$ $ c_i\geq 0 $ - suppose for a contradiction that $ c_i<0$; then $$\iota (x_{i-1}) = b_{i-1}/a_{i-1} <b_{i-1}/(a_{i-1}-1) \leq  (b_{i-1}+d_i)/(a_{i-1}+c_i) = \iota(x_i) <r$$ (note that we cannot have $a_{i-1}=1$ if $c_i<0$) and so $ (a_{i-1}-1)h_0+b_{i-1}h_\infty \in S_{i-1} $ - contradicting our choice of $ x_{i-1}$;

$\bullet$ $ d_i/c_i>\iota(x_{i-1}) $ - since $ d_i/c_i\leq \iota(x_{i-1}) $ would imply that: $$b_i/a_i = (b_{i-1}+d_i)/(a_{i-1}+c_i) \leq b_{i-1}/a_{i-1}$$ (by \ref{iotaest}(a)), which contradicts the fact that $ x_i \in S_i$.

\noindent Therefore $d_i/c_i > \iota(x_0)$ and so it follows from our choice of $ x_0$, that $ d_i/c_i>r$.

Therefore each $c_ih_0+d_ih_\infty$ belongs to the finite set: $$U= \{ y\in {\rm rad}^+(\chi ): \iota(y)> r \mbox{ and } \mu (y)\leq k\}$$ and define, for all $ n$: $$U_n=\{ \sum_{i=1}^n y_i: y_i\in U \mbox{ for all } i\leq n\}.$$ By construction,  $x_i\in \{x_0+z: z\in U_i\} $ for all $ i\geq 1$. We claim that there exists $ n $ such that: $$\iota(x_0+z)>r \mbox{ for all } z\in U_n$$ - this will give our required contradiction.

To prove this, let $ z_0\in U $ be minimal in this set with respect to the ordering on ${\rm rad}^+(\chi )$. Let $ e_0, f_0\in {\mathbb N} $ be such that $ e_0h_0+f_0h_\infty =z_0$; note that $f_0 >0$.  By repeated application of \ref{iotaest}(a), $\iota(z_0)\leq \iota(z)$ for all $z\in U_n$.

By \ref{iotaest}(b) there is $ N $ such that $ \iota(x_0+Nz_0)>r$. Take any $ z=eh_0+fh_\infty \in U_{Nf_0}$. Then $ f\geq Nf_0$. Let $ q=f/(Nf_0)\geq 1$. Notice that: $$r<\iota(x_0+Nz_0)= \dfrac{b_0+Nf_0}{a_0+Ne_0} \leq  \dfrac{b_0+qNf_0}{a_0+qNe_0}$$ (since $ (b_0+Nf_0)/(a_0+Ne_0)\leq f_0/e_0 = (qN-N)f_0/(qN-N)e_0$ by \ref{iotaest}(a), then apply \ref{iotaest}(a) again). Also: $$ \dfrac{f_0}{e_0} \leq \dfrac{f}{e} =\dfrac{qNf_0}{e}.$$ So $ e\leq qNe_0 $ and so: $$r< \iota(x_0+Nz_0) \leq \dfrac{b_0+qNf_0}{a_0+qNe_0} \leq \dfrac{b_0+f}{a_0+e} = \iota(x_0+z)$$ - so $ \iota(x_0+z)>r $ for all $ z\in U_{Nf_0} $ - thus proving the claim and hence establishing the contradiction.
\end{proof}

If ${\mathcal T}(\rho)$ is a tube then we use $n_\rho$ to denote the number of quasisimples in ${\mathcal T}(\rho)$.

\begin{lemma}\label{nravg}\marginpar{nravg} Let $ C $ be one of the specified algebras and let $ {\mathcal T}(\rho ) $ be a nonhomogeneous tube in $ {\mathcal T}_\gamma  $ where $ \gamma \neq 0,\infty $. Write $ \gamma =b/a $ with $a,b$ coprime positive integers and suppose that $ b>|n_\rho \langle h_\infty ,y\rangle |$ for all $y\in \Omega$. Fix $ p $ such that
$$ \big| \dfrac{1}{\langle h_0,h_\infty \rangle} \big(\mu (\langle h_\infty ,y\rangle h_0-\langle h_0,y\rangle h_\infty\big)  +\mu(y)) \big| \leq p  \mbox{  for all } y\in \Omega.$$ (Note that $p$ is independent of $\gamma$.)  Let $E $ be any quasisimple in $ {\mathcal T}(\rho )$. Then:
$${\rm dim}(E)\geq \dfrac{1}{\langle h_0,h_\infty \rangle} \mu (ah_0+bh_\infty )-p  \hspace{1cm} (1).$$ Furthermore, if $ n_\rho =\langle h_0,h_\infty \rangle  $ then: $$\big| {\rm dim}(E)-\dfrac{1}{\langle h_0,h_\infty \rangle} \mu (ah_0+bh_\infty )\big|\leq p \hspace{1cm} (2).$$
\end{lemma}
\begin{proof} Let $ E_1,\dots ,E_{n_\rho}  $ denote the quasisimples of $ {\mathcal T}(\rho )$.  Since these have endomorphism ring $k$ and no self-extensions, their dimension vectors $x_i$ satisfy $\chi_C(x_i) =1$ so, by \ref{bijomega} there exists, for each $i$, $ c_i,d_i\in {\mathbb Z} $ and $ y_i\in \Omega $ such that: $$\underline{\rm dim}(E_i) =c_ih_0+d_ih_\infty +y_i.$$ Since the slope of $ E_i $ is $ b/a $ we have: $$b/a = \iota (\underline{\rm dim}(E_i)) = \dfrac{d_i\langle h_0,h_\infty \rangle +\langle h_0,y_i\rangle }{c_i\langle h_0,h_\infty \rangle -\langle h_\infty ,y_i\rangle }.$$

Let $ k_i={\rm gcd}\big(d_i\langle h_0,h_\infty \rangle +\langle h_0,y_i\rangle , c_i\langle h_0,h_\infty \rangle -\langle h_\infty ,y_i\rangle \big) $ (noting that both terms are non-zero). Then: $$k_ib=d_i\langle h_0,h_\infty \rangle +\langle h_0,y_i\rangle $$ $$k_ia=c_i\langle h_0,h_\infty \rangle -\langle h_\infty ,y_i\rangle .$$

So: $$\underline{\rm dim}(E_i) =\dfrac{1}{\langle h_0,h_\infty \rangle} \big((k_ia+\langle h_\infty ,y_i\rangle )h_0+(k_ib-\langle h_0,y_i\rangle) h_\infty \big) +y_i \hspace{0.5cm} (\ast).$$  It follows from \cite[5.3.3]{RinTame} (see \cite[\S 4.1]{Har}) that $ ah_0+bh_\infty =\sum_{i=1}^{n_\rho } \underline{\rm dim}(E_i)$. By considering the last coordinate in $ {\mathbb Z}^n $ of this equation we get (for each type of algebra $C$):
$$b=\sum_{i=1}^{n_\rho } d_i =\sum _{i=1}^{n_\rho } \dfrac{k_ib-\langle h_0,y_i\rangle }{\langle h_0,h_\infty \rangle }.$$ So:
$$ b\Big( \langle h_0,h_\infty \rangle -\sum _{i=1}^{n_\rho } k_i \Big) = -\sum _{i=1}^{n_\rho } \langle h_0 ,y_i\rangle .$$
Since $b> |n_\rho \langle h_\infty ,y\rangle|$ for all $ y\in \Omega$, we must have: $$\langle h_0,h_\infty \rangle =\sum_{i=1}^{n_\rho } k_i.$$

Recall that we are proving two statements. For the second, (2), if $ n_\rho =\langle h_0,h_\infty \rangle $, then $ k_i=1 $ for all $ i\leq k $ so: $$d_i=\dfrac{1}{n_\rho }(b-\langle h_0,y_i\rangle )$$
$$ c_i=\dfrac{1}{n_\rho }(a+\langle h_\infty ,y_i\rangle )$$ for all $ i\leq n_\rho $. Thus, using ($\ast$):
$$ \big| {\rm dim}(E_i) - \dfrac{1}{n_\rho }\mu (ah_0+bh_\infty )\big| = \big| \dfrac{1}{n_\rho }\big( \mu (\langle h_\infty ,y_i\rangle h_0-\langle h_0,y_i\rangle h_\infty\big)  +\mu(y_i)\big|\leq p$$ by choice of $p$.

To finish the proof of (1), if $ n_\rho < \langle h_0,h_\infty \rangle  $ then, again using ($\ast$):
$$ {\rm dim}(E_i) -\dfrac{1}{\langle h_0,h_\infty \rangle} \mu (ah_0+bh_\infty ) =$$ $$= \dfrac{1}{\langle h_0,h_\infty \rangle} \big( \mu ((k_ia+\langle h_\infty ,y_i\rangle )h_0) +\mu( (k_ib-\langle h_0,y_i\rangle )h_\infty) - \mu (ah_0+bh_\infty )\big) +\mu(y_i) \geq $$ $$ \geq  \dfrac{1}{\langle h_0,h_\infty \rangle}  \big( \mu ((a+\langle h_\infty ,y_i\rangle )h_0 ) +\mu ((b-\langle h_0,y_i\rangle )h_\infty) - \mu (ah_0+bh_\infty )\big) +\mu(y_i)=$$ $$= \dfrac{1}{\langle h_0,h_\infty \rangle} \big(\mu(\langle h_\infty,y_i\rangle h_0) -\mu (\langle h_0,y_i\rangle h_\infty ) +\mu(y_i) \geq -p$$ by choice of $p$, as required.
\end{proof}

It would be possible to check the conclusion of \ref{nravg} for these algebras just by computation: finding the dimension vectors of inhomogeneous quasisimple modules then computing the dimension vectors of their $\tau$-translates.  This, however, involves checking through many cases so we follow the proof in \cite{Har} since it covers all the relevant algebras with considerably less computation.  We do, though, need the calculation that, for each of these algebras, there is an inhomogeneous tube of rank $\langle h_0, h_\infty\rangle$, see \cite[\S 4.2]{Har}.

\begin{theorem}\label{dimleap}\marginpar{dimleap} Let $ C $ be of the form $ C(4,\lambda )$, $ C(6)$, $ C(7)$ or $ C(8)$. Given any positive irrational $ r$, any $ \epsilon >0 $ and any $ d\geq 1$, there exists a tube $ {\mathcal T}(\rho ) $ of rank $ \langle h_0,h_\infty \rangle  $ and with slope in $ (r-\epsilon ,r) $ such that, if $ E $ is a quasisimple of $ {\mathcal T}(\rho ) $ and $ N $ is an indecomposable finite-dimensional $ C$-module, then $$\iota(\underline{\rm dim}(N))\in (\iota(\underline{\rm dim}(E)),r) \mbox{ implies } {\rm dim}(N)\geq {\rm dim}(E)+d.$$ In particular, any nonzero morphism from $ E $ to a module in $ (\iota\underline({\rm dim}(E)) ,r) $ is an embedding.
\end{theorem}

\begin{proof} Let $ p $ be the bound from \ref{nravg}. Pick any $ k\geq 1 $ large enough such that: $$\dfrac{1}{\langle h_0,h_\infty \rangle} k-2p\geq d.$$

By \ref{jumpnrr} there exist coprime $ a,b\in {\mathbb N} $ such that $ r-\epsilon <b/a <r $ and, given any $ a',b'\in {\mathbb N}$, $$b/a<b'/a'<r \Rightarrow \mu (a'h_0+b'h_\infty )>\mu (ah_0+bh_\infty )+k.$$ Pick $ {\mathcal T}(\rho ) $ to be any tube of index $ b/a $ and rank $ \langle h_0,h_\infty \rangle $.

Now, take any indecomposable $ N\in {\rm mod}\mbox{-}C $ with slope in the interval $ (b/a,r)$. We have $ N\simeq E'[j] $ for some quasisimple $ E'$. Let $ a',b' $ be coprime integers such that $ b'/a' $ is the slope of $ E'$. By \ref{jumpnrr} and \ref{nravg}(1) then \ref{nravg}(2) (we may assume that our choice of $b/a$ is such that every $b'$ occurring satisfies the condition on $b$ in \ref{nravg}): $${\rm dim}(E'[j])\geq {\rm dim}(E')\geq  \dfrac{1}{\langle h_0,h_\infty \rangle} \mu (a'h_0+b'h_\infty ) -p \geq$$ $$\geq  \dfrac{1}{\langle h_0,h_\infty \rangle} \big((\mu (ah_0+bh_\infty) +k\big)-p \geq $$ $$\geq  \dfrac{1}{\langle h_0,h_\infty \rangle} (\langle h_0,h_\infty \rangle \,{\rm dim}(E)-\langle h_0,h_\infty \rangle p+k)-p \geq$$ $$ \geq {\rm dim}(E)+d.$$
\end{proof}

\section{Width is undefined}\label{secwide}\marginpar{secwide}

Still $ C $ is an algebra of one of the types specified in Section \ref{sechomog}.  We show that the width of the lattice ${\rm pp}/\sim_r$ of definable subgroups of modules of slope $r$ is undefined.  This will be extended to arbitrary tubular algebras in the next section.

Recall that if $r$ is a positive irrational then we write $\phi \sim_r \psi$ iff $\phi(M(r)) =\psi(M(r))$ where $M(r)$ is any module which generates ${\mathcal D}_r$ as a definable category.  Also recall that a pp-pair is said to be open at $r$ if $\phi$ and $\psi$ are not equivalent under the equivalence relation $\sim_r$ and this, see \ref{cofinrtl}, is equivalent to being open near the left (or near the right) of $r$.  In \cite{Har} the relation $\sim_r$ was defined purely in terms of finite-dimensional modules, as being closed near the left of $r$ but, in view of \ref{cofinrtl}, these are equivalent.

\begin{theorem}\label{Cwidth}\marginpar{Cwidth} Let $ C $ be of the form $ C(4,\lambda )$, $ C(6)$, $ C(7)$ or $ C(8)$. For each positive irrational $ r $ the lattice $ {\rm pp}_C/\sim _r$, that is, the lattice of pp-definable subgroups of any generator $ M(r) $ for the definable category $ {\mathcal D}_r $,  has width $ \infty $ (that is, is undefined).  Indeed, every non-trivial interval in the lattice ${\rm pp}_C/\sim_r$ contains incomparable elements.
\end{theorem}

\begin{proof} Take any pp-pair $ \phi /\psi  $ which is open at $ r$. By \ref{ppashom} and \ref{ppashominf} we may replace each of the formulas $ \phi , \psi  $ by a formula to which it is $ \sim _r$-equivalent and with free realisations, let us write them as $ (M,m) $ and $ (N,n) $ (rather than $M'$ etc.), with properties as in \ref{ppashom} for some $\epsilon>0$.  We assume that we have made these replacements, so continue to write $ \phi , \psi $.  We may also suppose, by \ref{allopenlh}, that $\epsilon$ is small enough that $\phi/\psi$ is open on every module in $(r-\epsilon, r)$.

Let $ d={\rm dim}(N) $ and apply \ref{dimleap} to obtain $ \gamma \in (r-\epsilon ,r) \cap {\mathbb Q} $ and a tube $ {\mathcal T}(\rho ) $ of index $ \gamma  $ as in the statement of that result. Pick any quasisimple module $ E $ in $ {\mathcal T}(\rho ) $ and let $ E' $ be any other quasisimple module in that tube. Fix $ x\in \phi (E)\setminus \psi (E) $ and $ x'\in \phi (E')\setminus \psi (E')$ and let $ \theta $, respectively $ \theta '$  generate the pp-type of $ x $ in $ E$, resp. of $ x' $ in $ E'$. We shall show that the images of $ \psi +\theta  $ and $ \psi +\theta ' $ in $ {\rm pp}_C/\sim _r $ are incomparable.

$$
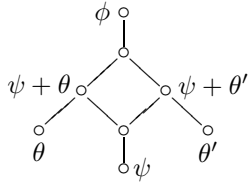
\vcenter{%
\xymatrix@C=10pt@R=8pt{%
&&*+={\circ}\ar@{-}[d]\ar@{}+<-8pt,0pt>*{\phi}&&\\
&&*+={\circ}&&\\
&*+={\circ}\ar@{-}[ur]\ar@{}+<-16pt,+2pt>*{\psi+\theta} &&*+={\circ}\ar@{-}[ul]\ar@{}+<+18pt,+2pt>*{\psi+\theta'}&\\
*+={\circ}\ar@{-}[ur]\ar@{}+<-0pt,-8pt>*{\theta}&&*+={\circ}\ar@{-}[d]\ar@{-}[ul]\ar@{-}[ur]\ar@{} &&
*+={\circ}\ar@{-}[ul]\ar@{}+<0pt,-8pt>*{\theta'}\\
&&*+={\circ}\ar@{}+<+7pt,0pt>*{\psi}
}}
$$

So suppose, for a contradiction, that $ \psi +\theta \leq _r \psi +\theta '$, that is, $ \theta \sim _r\theta \wedge (\psi +\theta ')$.  Therefore, by \ref{cofinrtl}, there is $ \delta >0 $ such that for all $X $ in $ (r-\delta ,r) $ we have $ \theta (X)=(\theta \wedge (\psi +\theta '))(X)$; we may take $ \delta <\epsilon $. A free realisation for $ \theta \wedge (\psi +\theta ') $ may be obtained, by \ref{ppmeetsum}, as the element $ l=g(x)=g'(n,x') $ in the pushout module $ L $ shown.

$\xymatrix{C \ar[r]^x \ar[d]_{(n,x')} &  E \ar[d]^g \\ N\oplus E' \ar[r]^{g'} &  L}$

\noindent Note that $ {\rm dim}(L)<{\rm dim}(E)+{\rm dim}(E')+{\rm dim}(N)$.

Write $L$ as $L'\oplus L''$ where each summand of $L'$ has slope $<r$ and each summand of $L''$ has slope $>r$.  Set $f'=\pi'g$ where $\pi':L\rightarrow L'$ is the induced projection.  Suppose that $ h : E \rightarrow Z $ where $Z$ is in $ (r-\delta ,r) $; then $h$ factors through $ g$. Indeed, since $ h(x)\in \theta (Z) \leq  (\psi +\theta ')(Z)$, there must exist, by \ref{freereal}, a map $ h':N\oplus E' \rightarrow Z $ with $h'(n,x')= h(x)$. The pushout property then gives a factorisation of $h$ through $g$, and hence through $f'$.  Since there are, by \ref{canonmodcat}, nonzero morphisms from $E$ to modules $Z$ in $(r-\delta, r)$, we deduce that $f'(x)\neq 0$ (in particular $L'\neq 0$).  We will show that $L'$ has no summand of slope $>\gamma$; we do this by showing that this is also true of ${\rm coker}(f')$, noting that by choice of $\gamma$ and since $L'$ is in $(\gamma, r)$, $f'$ is, by \ref{dimleap}, an embedding, so we have the exact sequence $$0\rightarrow  E\xrightarrow{f'}  L' \rightarrow  {\rm coker}(f') \rightarrow 0.$$   Pick $\gamma' \in (\gamma, r)$ such that $L'$ is in $(0,\gamma')$ (indeed, in $[\gamma, \gamma')$). First we show that ${\rm coker}(f')$ is in $[0,\gamma ')$.

Pick $Z$ of slope $\gamma'$; then we have the long exact sequence $$0\rightarrow  ({\rm coker}(f'),Z)\rightarrow  (L',Z)\xrightarrow{(f',Z)} (E,Z)\rightarrow  {\rm Ext}({\rm coker}(f'),Z)\rightarrow  {\rm Ext}(L',Z)=0$$ (the last term is $0$ since, by \ref{AusReit}, it has the same dimension as $(\tau^{-1}Z,L')=0$).  We have just seen that $(f',Z)$ is surjective so it follows that ${\rm Ext}({\rm coker}(f'),Z)=0$ and then, by \ref{AusReit} and \ref{canonmodcat}, that $(\tau^{-1}Z, {\rm coker}(f'))=0$.  This is so for every $Z$ of slope $\gamma'$ so, by \ref{canonmodcat}, ${\rm coker}(f')$ must be in $[0,\gamma')$.

Now, notice that ${\rm dim}({\rm coker}(f'))={\rm dim}(L')-{\rm dim}(E) \leq  {\rm dim}(L)-{\rm dim}(E) < {\rm dim}(N)+{\rm dim}(E') + {\rm dim}(E) -{\rm dim}(E) = {\rm dim}(N)+{\rm dim}(E')$.  So, by choice of $\gamma$, ${\mathcal T}(\rho)$ to satisfy \ref{dimleap}, and the fact that $E'$ is in ${\mathcal T}(\rho)$, it must be that every summand of ${\rm coker}(f')$ with slope $>\gamma$ has slope $>r$.  We saw above that every summand of ${\rm coker}(f')$ has slope $<\gamma'<r$, so we deduce that ${\rm coker}(f')$ has slope $\gamma$.

If $f'E$ lies in a complement of an indecomposable summand of $L'$ then we can drop that summand from the second and third terms of the sequence, hence we can assume that each (remaining) summand of $L'$ has slope at least $\gamma$.  If such a summand has nonzero image in ${\rm coker }(f')$ then its slope is, by the above, exactly $\gamma$.  If it had zero image then it would have to be contained in $f'E\simeq E$ hence, since $E$ is quasisimple, would equal $f'E$ so, again, would have slope $\gamma$.  So we have an exact sequence (perhaps after dropping some unnecessary summands) $$0\rightarrow  E\xrightarrow{f'}  L' \rightarrow  {\rm coker}(f') \rightarrow 0$$ lying in ${\rm add}({\mathcal T}(\rho))$.  If $L'$ is not indecomposable then we can replace it by a direct summand to which the induced map from $E$ is nonzero hence still an embedding, and that sequence will, up to isomorphism, have the form $$0\rightarrow E \xrightarrow{f''} E[k] \xrightarrow{p} \tau^{-1}E[k-1] \rightarrow 0.$$

Since $E[k]$ has quasisimple socle $E$, it follows that $\pi_1g'E'=0$ ($g'$ as in the diagram at the start of the proof), where $\pi_1:L\rightarrow E[k]$ is the projection, and hence $\pi_1g'(n)=f''(x)$, so $p \pi_1 g'(n)=0$.  Therefore $p\pi_1g'$ factors through ${\rm coker}(n)$ which, by choice of $N$, $n$ and $\epsilon$ to satisfy \ref{ppashom}, must be in $(r,\infty]$.  Therefore $p\pi_1g'=0$ and $\pi_1g'$ takes $N$ to $f''E$.  Since $n\in \psi(N)$ we deduce that $f''(x)\in \psi(f''E)$ and hence, since $f''$ is an isomorphism from $E$ to its image, that $x\in \psi(E)$ - a contradiction to the choice of $x$.

Thus $ \theta +\psi  $ and $ \theta '+\psi  $ are incomparable as claimed. This is true for arbitrary pairs $ \phi /\psi  $ as at the start of the proof, so we have shown, in particular, that the width of this lattice is $ \infty $.
\end{proof}

\section{Moving to other tubular algebras}\label{secmove}\marginpar{secmove}

To extend to arbitrary tubular algebras we use that, given such an algebra $ A $ there is a shrinking functor from $ A\mbox{-}{\rm Mod} $ to $ B\mbox{-}{\rm Mod} $ where $B$ is a canonical tubular algebra and then, given any canonical tubular algebra $B$, there is a shrinking functor from $ B\mbox{-}{\rm Mod} $ to $ C\mbox{-}{\rm Mod} $ where $ C $ is an algebra of one of the types dealt with in the previous sections. A {\bf shrinking functor} is a certain type of tilting functor, namely one where the tilting module $ T $ has the form $ T_0\oplus T_p $ where $ T_0 $ is a preprojective tilting $ A_0$-module, regarded as an $ A$-module ($ A_0 $ is the support of the radical vector $h_0$ and is obtained by omitting the single source vertex from the quiver of $ A$) and $ T_p $ is a projective $ A$-module - we say that $ T $ is a {\bf shrinking module}.  We recall the equivalences induced by tilting functors.

If $ T $ is a tilting left $ A$-module, set $ B={\rm End}(T) $ and consider the functors $ (T,-), {\rm Ext}^1_A(T,-):A\mbox{-}{\rm Mod} \rightarrow  B\mbox{-}{\rm Mod} $ and $ (T\otimes _B -), {\rm Tor}^B_1(T,-): B\mbox{-}{\rm Mod} \rightarrow  A\mbox{-}{\rm Mod}$. Define the subclasses $ {\mathcal F}(T)={\rm Ker}(T,-) $ and $ {\mathcal G}(T)={\rm Ker}({\rm Ext}(T,-)) $ of $ A\mbox{-}{\rm Mod} $ and the subclasses $ {\mathcal X}(T)={\rm Ker}(T\otimes -) $ and $ {\mathcal Y}(T)={\rm Ker}({\rm Tor}(T,-)) $ of $ B\mbox{-}{\rm Mod} $ (these are definable subclasses of the respective module categories, see \cite[4.8]{PreInterp} or \cite[\S\S 10.2.6, 18.2.3]{PreNBK}). The fundamental theorem gives equivalences between these.  Recall that $({\mathcal F}, {\mathcal T})$ is said to be a {\bf torsion pair} if $({\mathcal T}, {\mathcal F})=0$ and if ${\mathcal F}$, the {\bf torsionfree class}, and  ${\mathcal T}$, the {\bf torsion class} are maximal with respect to this property (we follow \cite{ReiRin} in writing the torsionfree class on the left).

\begin{theorem}\label{tilt}\marginpar{tilt} (\cite{BreBut})) If $ T $ is a tilting module then, with notation as above, $ ({\mathcal F}(T), {\mathcal G}(T)) $ is a torsion pair in $ A\mbox{-}{\rm Mod} $ and $ ({\mathcal Y}(T), {\mathcal X}(T)) $ is a torsion pair in $ B\mbox{-}{\rm Mod}$. Furthermore $ (T,-) $ and $ T\otimes - $ are inverse equivalences between $ {\mathcal G}(T) $ and $ {\mathcal Y}(T) $, and $ {\rm Ext}^1(T,-) $ and $ {\rm Tor}_1(T,-) $ are inverse equivalences between $ {\mathcal F}(T) $ and $ {\mathcal X}(T)$.
\end{theorem}

Suppose now that $ A $ is a tubular algebra and that $ T $ is a shrinking $ A$-module; set $ B={\rm End}(T)$. Then (\cite[5.5.1]{RinTame}) $ B $ also is tubular and (\cite[4.1.7]{RinTame}) there is a linear map $ \sigma _T:K_0(A)\rightarrow K_0(B) $ such that $ \sigma _T(\underline{{\rm dim}}(M)) = \underline{{\rm dim}}(T,M) - \underline{{\rm dim}}({\rm Ext}(T,M))$. By \cite[p.~290]{RinTame}, $ \sigma _T(h^A_0)=h^B_0 $ and $ \sigma _T(h^A_\infty ) $ is in the radical of $ \chi_B$, so there exist non-negative rationals $ n_0,n_\infty  $ such that $ \sigma _T(h^A_\infty )=n_0h^B_0+n_\infty h^B_\infty $. Define $ \overline{\sigma }:{\mathbb Q}^\infty _0\rightarrow {\mathbb Q}^\infty _0 $ by $ \overline{\sigma }(\gamma )=\dfrac{n_\infty \gamma }{n_0\gamma +1}$; then $\overline{\sigma}(\infty)= n_\infty / n_0$ and $ \overline{\sigma } $ is an order-preserving bijection from $ {\mathbb Q}^\infty _0 $ to $ \{\delta \in {\mathbb Q}^\infty _0: 0\leq \delta \leq n_\infty /n_0\}$, so we can also denote by $ \overline{\sigma } $ the extension of this map to all real numbers in the respective intervals. The tilting functor $ (T,-) $ is said to be {\bf proper} if $ n_0\neq 0$, that is, if $\overline{\sigma}(\infty)\neq \infty$.

\begin{theorem}\label{tiltequiv}\marginpar{tiltequiv} \cite[5.4.1, 5.4.2', 5.4.3]{RinTame} Suppose that $ T $ is a shrinking $A$-module. Then any indecomposable finitely presented module not in $ {\mathcal G}(T) $ is a preprojective $ A_0$-module.  If $(T,-)$ is proper then this functor is an equivalence between $ {\mathcal P}^A_\infty  \cap {\mathcal G}(T) $ and $ {\mathcal P}^B_{\overline{\sigma}(\infty)} $. Furthermore, for each $ \gamma \in {\mathbb Q}^\infty _0$, $ (T,-) $ induces an equivalence between $ {\mathcal T}_\gamma  $ and $ {\mathcal T}_{\overline{\sigma }(\gamma )}$.
\end{theorem}

Recall that $ M(r) $ denotes any module which generates the definable subcategory of $ A$-modules with slope $ r$; similarly denote by $ N(\overline{\sigma }(r)) $ a $ B$-module which generates the definable category of modules of slope $ \overline{\sigma }(r)$. We aim to show that if the width of $ {\rm pp}^B/\sim_{\overline{\sigma }(r)} $ is $ \infty  $ then so is the width of $ {\rm pp}^A/\sim_r$.  We could use the very general results about interpretation functors in \cite{PreDefAddCat} - as we do in Section \ref{secfurther}, see \ref{locequiv}.  But we can give a direct and explicit proof as follows.

Suppose that $ \overline{t}=(t_1,\dots, t_k) $ is a sequence of elements which together generate $ T $ as an $ A$-module. We define a map, easily seen to be order-preserving, from $ {\rm pp}_B $ to $ {\rm pp}_A^k $ as follows, where $ {\rm pp}_A^k $ denotes the lattice of equivalence classes of pp formulas with $k$, rather than just one, free variables (see the background references for details). If $ \phi (v)\in {\rm pp}_B$, let $ (M,m) $ be a free realisation of $ \phi  $ and consider the $ k$-pointed $ A$-module $ (T\otimes M,\overline{t}\otimes m) $ (where $ \overline{t}\otimes m $ means $ (t_1\otimes m, \dots, t_k\otimes m)$); let $ \phi ^T $ denote a pp formula which generates the pp-type of the $k$-tuple $ \overline{t}\otimes m $ in $ T\otimes M$.

\begin{theorem}\label{widthatr}\marginpar{widthatr} If $ \phi ,\psi \in {\rm pp}_B $ are such that $ \phi ^T\sim_r \psi ^T $ then $ \phi \sim_{\overline{\sigma }(r)} \psi  $, hence the map $ \phi \mapsto \phi ^T $ induces an embedding of $ {\rm pp}_B/\sim_{\overline{\sigma }(r)} $ into $ {\rm pp}_A/\sim_r$. In particular if $ {\rm w}({\rm pp}_B/\sim_{\overline{\sigma }(r)})=\infty  $ then $ {\rm w}({\rm pp}_A/\sim_r)=\infty $.
\end{theorem}

\begin{proof} Let $ (M,m) $ and $ (N,n) $ be free realisations of $ \phi  $ and $ \psi  $ respectively. Suppose that $ \phi ^T\sim_{r} \psi ^T$. Pick any $ \epsilon >0 $ such that $ \phi ^T(X) =\psi ^T(X) $ for all $ X $ in $ (r-\epsilon ,r)$ (by \ref{allopenl1} there is such an $\epsilon$). We claim that $ \phi (Y)=\psi (Y) $ for all $ Y $ in $ (\overline{\sigma }(r-\epsilon ),\overline{\sigma }(r))$ and hence (by \ref{allopenl1}) that $\phi \sim_{\overline{\sigma }(r)} \psi $.

Take any $ y\in \phi (Y)$, so there is $ h\in (M,Y) $ with $ h(m)=y$. Consider the map $ T^k \xrightarrow{-\otimes m} T\otimes M \xrightarrow{1_T\otimes h} T\otimes Y $, where the first $A$-linear map is defined by $t_i\mapsto t_i\otimes m$, and set $ x_i =t_i\otimes y$ to be the image of $ t_i $ in $T\otimes Y$ for $ i=1,\dots, k$. Then $ \overline{x} =(x_1,\dots, x_k) \in \phi ^T(T\otimes Y) $ (by definition of $ \phi ^T$ and \ref{freereal}). It follows from what was said above that $ T\otimes Y $ is in $ (r-\epsilon ,r) $ so, by assumption, $ \overline{x}  \in \psi ^T(T\otimes Y) $, therefore there is a map $ f:T\otimes N\rightarrow T\otimes Y $ taking $ t_i\otimes n $ to $ x_i $ for each $ i$.

By \ref{tiltequiv} we can apply the inverse equivalence $ (T,-) $ from \ref{tilt}.  Then we have an isomorphism $N\simeq (T,T\otimes N)$ which takes $n$ to the map $t\mapsto t\otimes n$; the composition with $(T,f):(T,T\otimes N) \rightarrow (T,T\otimes Y) \simeq Y$ takes $n$ to the map which takes $t_i$ to $f(t_i\otimes n) =x_i =t_i\otimes y$ and that is the map which corresponds to $y$ under the isomorphism $Y\simeq (T,T\otimes Y)$.  That is, we have a map $N\rightarrow Y$ which takes $n$ to $y$.  Since $n\in \psi(N)$ it follows that $y\in \psi(Y)$, as required.
\end{proof}

Combined with \ref{spdecwdth} this gives the result on width and superdecomposable pure-injectives.
\begin{cor}\label{spdecatr}\marginpar{spdecatr} Let $A$ be a tubular algebra and let $r$ be a positive irrational.  Then the width of the lattice of pp formulas for the definable category ${\mathcal D}_r$ is undefined.  If $A$ is countable then there is a superdecomposable pure-injective module of slope $r$.
\end{cor}

Another dimension was introduced in \cite[p.~191]{Zie}; in \cite[\S 10.2]{PreBk}, it was termed m-dimension and given a definition in terms of inductively collapsing intervals of finite length in the lattice of pp formulas.  What turned out to be the same dimension, but set in the associated functor category and defined by inductively collapsing finitely presented functors of finite length, was introduced by Geigle in \cite{GeiKGdim} (also see \cite[p.~197ff.]{JeLe}) and termed Krull-Gabriel dimension; we will use this latter term.  Note that it refers to a process not in the definable category but in the associated functor category (see Section \ref{secfurther} and, e.g., \cite[\S 13.2]{PreNBK}).

\begin{cor}\label{KGatr}\marginpar{KGatr} Let $A$ be a tubular algebra and let $r$ be a positive irrational.  Then the Krull-Gabriel dimension of the definable category ${\mathcal D}_r$ is undefined.  If $A$ is countable then there are continuum many indecomposable pure-injective modules of slope $r$.
\end{cor}

The consequence in the case that the field is countable is another theorem of Ziegler (\cite[8.1, 8.4]{Zie}).

Both \ref{spdecatr} and \ref{KGatr} (without any restrictions on cardinality) are statements about the complexity of morphisms in the category of {\bf pointed} finitely presented {\bf modules} $(M,m)$, equivalently morphisms $A\rightarrow M$.  These may be preordered by the relation $(M,m)\geq (N,n)$ if there is a morphism $f:M\rightarrow N$ with $f(m)=n$ and the resulting poset is naturally identified with ${\rm pp}_A$.  (More generally, if $X\in A\mbox{-}{\rm mod}$, then the similarly defined poset of $X$-pointed modules, that is morphisms $X\rightarrow Y\in A\mbox{-}{\rm mod}$ is, if $X$ is $k$-generated, a sublattice of ${\rm pp}^k_A$, see \cite[\S 3.1]{PreNBK}.)  Krull-Gabriel dimension being undefined is equivalent to existence of a `factorisable system of morphisms' in $A\mbox{-}{\rm mod}$ as defined in \cite{PreMor} (or see \cite[7.2.13]{PreNBK}).  It is less easy to give a nice description, in terms of the category of finitely presented modules, of width being undefined but a useable sufficient condition is given in \cite[5.4]{PyaPunToff3}.

\section{Further extensions}\label{secfurther}\marginpar{secfurther}

We obtain more information about the definable categories of modules supported at an irrational or supported on an open interval.  For some results we will have to use the results proved for the particular algebras $C$ from Section \ref{sechomog} and then use tilting functors, as in Section \ref{secmove}, to obtain the general case.  In order to facilitate that, we will use the general results on interpretations mentioned (and indeed illustrated in a particular case) in the previous section.

\begin{cor}\label{locequiv}\marginpar{locequiv} Suppose that $T$ is a shrinking $A$-module such that $(T,-)$ is proper and let $r$ be a positive irrational.  Then there is an open neighbourhood $(r-\epsilon, r+\epsilon)$ such that the restrictions of $(_AT,-)$ and $T\otimes_B -$ are equivalences between the definable subcategory ${\mathcal D}_{(r-\epsilon, r+\epsilon)}$ of $A\mbox{-}{\rm Mod}$ of modules supported on $(r-\epsilon, r+\epsilon)$ and the definable subcategory ${\mathcal D}_{(\overline{\sigma}(r-\epsilon), \overline{\sigma}(r+\epsilon))}$ of $B\mbox{-}{\rm Mod}$ of $B$-modules supported on the open neighbourhood $(\overline{\sigma}(r-\epsilon), \overline{\sigma}(r+\epsilon))$ of $\overline{\sigma}(r)$, the equivalences being given by the restrictions of $(_AT,-)$ and $T\otimes_B-$.
\end{cor}
\begin{proof} Set ${\mathcal D}={\mathcal D}_{(r-\epsilon, r+\epsilon)}$  and ${\mathcal E}={\mathcal D}_{(\overline{\sigma}(r-\epsilon), \overline{\sigma}(r+\epsilon))}$.  By \ref{tiltequiv} and the discussion that precedes that, these functors define inverse equivalences between the categories of finitely presented modules in these categories.  We may assume that $r-\epsilon$ is rational (and $>0$) and, therefore (just by its definition), so is $\overline{\sigma}(r-\epsilon)$.  By \ref{ddplus} each module in ${\mathcal D}$ is a direct limit of finite-dimensional modules in ${\mathcal D}$, and the same for ${\mathcal E}$.  That is, each of these definable categories is locally finitely presented and any functor between them which commutes with direct limits, as both $(T,-)$ and $T\otimes -$ do, is determined by its restriction to the finite-dimensional modules.  Since the compositions of these two functors are naturally equivalent to the respective identity functors on finite-dimensional modules, they are inverse equivalences as claimed.
\end{proof}

Since both $(T,-)$ and $T\otimes-$ also commute with direct products they are, \cite[18.2.22]{PreNBK}, interpretation functors in the sense of \cite{PreDefAddCat}, see also \cite{KraEx}, and so induce an inclusion-preserving bijection between the definable subcategories of ${\mathcal D}_{(r-\epsilon, r+\epsilon)}$  and ${\mathcal D}_{(\overline{\sigma}(r-\epsilon), \overline{\sigma}(r+\epsilon))}$, from which we deduce the following.

\begin{cor}\label{locequivatr}\marginpar{locequivatr}  The inverse natural equivalences in \ref{locequiv} restrict to inverse natural equivalences between the subcategory, ${\mathcal D}_r$, of $A\mbox{-}{\rm Mod}$ of modules of slope $r$ and the subcategory, ${\mathcal D}_{\overline{\sigma}(r)}$, of $B\mbox{-}{\rm Mod}$ of modules of slope $\overline{\sigma}(r)$.
\end{cor}

By \cite[13.1]{PreDefAddCat} (or, better, \cite[2.3]{PreRajShv}), $(T,-)$ and $T\otimes-$ induce equivalences of the respective functor categories.  If ${\mathcal D}$ is a definable category then we denote by ${\rm fun}({\mathcal D})$ the associated abelian functor category; we give a few words of explanation but see the references for details.  If ${\mathcal D}$ is locally finitely presented (e.g., as noted above, the category of modules supported on an open interval is so), then ${\rm fun}({\cal D})$ is the category of finitely presented functors from finite-dimensional objects in ${\mathcal D}$ to $k\mbox{-}{\rm mod}$ (rather, the extensions of these to functors on all of ${\mathcal D}$ which commute with direct limits).  In general a definable category need not be locally finitely presented (e.g., ${\mathcal D}_r$ is not) but every definable category is a definable subcategory of a locally finitely presented category and the category ${\rm fun}({\mathcal D})$ is obtained as a localisation of the category of functors on the larger category.  It can also be obtained directly as the category of those functors from ${\mathcal D}$ to $k\mbox{-}{\rm Mod}$ which commute with direct products and direct limits, alternatively as the category whose objects are the pp-pairs (of pp formulas modulo equivalence on all object of ${\mathcal D}$) and pp-definable maps between these (the category of ``'pp-imaginaries").  In particular, the lattice of pp formulas for ${\mathcal D}$ is naturally identified with the lattice of finitely generated subfunctors of the forgetful functor.  Therefore we have the following strengthening of \ref{widthatr}.

\begin{cor}\label{equivatrpp}\marginpar{equivatrpp}  The inverse natural equivalences in \ref{locequivatr} induce natural equivalences ${\rm fun}({\mathcal D}_r) \simeq {\rm fun}({\mathcal D}_{\overline{\sigma}(r)})$; in particular  ${\rm pp}_A/\sim_r \,\simeq\, {\rm pp}_B/\sim_{\overline{\sigma}(r)}$
\end{cor}

We show now that there are no proper nonzero definable subcategories of ${\mathcal D}_r$ if $r$ is irrational.

\begin{prop}\label{allequivwk}\marginpar{allequivwk}  Let $C$ be one of the algebras dealt with in Section \ref{sechomog} and let $r$ be a positive irrational.  Suppose that $M\in {\mathcal D}_r$.  Then $M$ generates ${\mathcal D}_r$ as a definable category.

In particular, ${\mathcal D}_r$ has no proper nonzero definable subcategories.
\end{prop}
\begin{proof} It has to be shown that if $M\in {\mathcal D}_r$ and if $\phi/\psi$ is a pp-pair which is open at $r$ then $\phi/\psi$ is open on $M$.  We use the fact that, since $C$ is finite-dimensional, if $N$ is any module and if $a\in \psi(N)$ then there is a submodule $N'$ of $N$ of length $< d$ ($d=1+{\rm dim}(R)$ times the number of variables in $\psi$ will do) such that $a\in \psi(N')$ (\cite[13.6]{PreBk} see the proof of \ref{cofinrtl}(iii)$\Rightarrow$(i)).

Choose $\epsilon>0$ such that the conclusion of \ref{allopenlh} is satisfied for $\phi/\psi$.  By \ref{limit}, $M$ is a union of submodules in $(r-\epsilon, r)$ and, by the factorisation property and \ref{dimleap}, there is a nonzero morphism $f:E[k] \rightarrow M$ where $E[k]$ is in a tube of slope $\gamma \in (r-\epsilon, r)$ which satisfies the conclusion of \ref{dimleap} for $d$ as above.  By factoring out the maximal $E[i]$ contained in ${\rm ker}(f)$ we may assume that $f$ restricted to the quasisimple socle, $E$, of $E[k]$, is not zero.  Since ${\rm im}(f\upharpoonright E)$ is finite-dimensional with indecomposable factors in $[\gamma,r)$ the last statement of \ref{dimleap} applies, so $f\upharpoonright E$ is monic.  Replace $E[k]$ by $E$ and choose $a\in \phi(E)\setminus \psi(E)$.

If we assume, for a contradiction, that $\phi/\psi$ is closed on $M$, then there is a finitely generated submodule $X$ of $M$ containing the image of $E$ in $M$ such that $f(a)\in \psi(X)$; using the fact from the beginning of the proof, we may assume that ${\rm dim}(X) < {\rm dim}(E)+d$.  Noting that each indecomposable summand of $X$ has slope in $[\gamma, r)$, it follows that the corestriction $E\rightarrow X$ of $f$ is a split embedding.  For otherwise it would induce a nonzero, non-isomorphic map from $E$ to an indecomposable summand of $X$, but that summand must have dimension $< {\rm dim}(E)+d$, in contradiction with choice of $\epsilon$ for \ref{dimleap}.  Therefore $E\rightarrow X$ is split and so, since $a\notin \psi(E)$, $a\notin \psi(X)$, a contradiction as required.
\end{proof}

By \ref{locequivatr} we can extend this to general tubular algebras.

\begin{theorem}\label{allequiv}\marginpar{allequiv}  Let $A$ be a tubular algebra and let $r$ be a positive irrational.  Then ${\mathcal D}_r$ has no proper nonzero definable subcategories.
\end{theorem}

\begin{cor}\label{sloperstrong}\marginpar{sloperstrong}  Let $A$ be a tubular algebra and let $r$ be a positive irrational.  Then every $A$-module $M$ in ${\mathcal D}_r$ satisfies $(M,X)\neq 0$ for every finite-dimensional $X$ in $(r,\infty)$ and $(X,M)\neq 0$ for every $X$ in $[0,r)$.
\end{cor}

The next result is an extension of \ref{cofinrtl} and \ref{allopenl1} to the general case.

\begin{cor}\label{allopenrtl}\marginpar{allopenrtl} Let $A$ be any tubular algebra. Let $ \phi /\psi  $ be any pp-pair and let $ r $ be a positive irrational. Then the following are equivalent:

\noindent (i) $\phi /\psi $ is open near the left of $ r$;

\noindent (ii) $\phi /\psi $ is open near the right of $ r$;

\noindent (iii) $\phi /\psi $ is open at $ r$, that is, open on some module with slope $r$;

\noindent (iv) $\phi /\psi $ is open at $ r$, that is, open on every (nonzero) module with slope $r$;

\noindent (v) there exists $ \epsilon >0 $ such that $ \phi /\psi  $ is open on every module in $ (r-\epsilon ,r)$;

\noindent (vi) there exists $ \epsilon >0 $ such that $ \phi /\psi  $ is open on every module in $ (r,r+\epsilon )$.

\noindent (vii) there exists $ \epsilon >0 $ such that $ \phi /\psi  $ is open on every module in $ (r-\epsilon,r+\epsilon )$.
\end{cor}

\begin{proof} By \ref{locequiv} there is an algebra $C$ of one of the forms $C(4,\lambda)$, $C(6)$, $C(7)$, $C(8)$ such that, if we fix a sufficiently small $\epsilon>0$, then the category of $A$-modules supported on $(r-\epsilon, r+\epsilon)$ is equivalent to the category of $C$-modules supported on the open interval $(\overline{\sigma}(r-\epsilon), \overline{\sigma}(r+\epsilon))$.  The restrictions of the relevant tilting functor $(T,-)$ and its inverse $T\otimes -$, to the definable subcategories ${\mathcal D}_{(r-\epsilon, r+\epsilon)}$ and ${\mathcal D}_{(\overline{\sigma}(r-\epsilon), \overline{\sigma}(r+\epsilon))}$ are, as noted above, interpretation functors.  Therefore, if $\phi/\psi$ is a pp-pair for $A$-modules then, by \cite[13.1, 25.4]{PreDefAddCat}, there is a corresponding pair $\phi'/\psi'$ of pp formulas for $C$-modules such that $\phi/\psi$ is open on the $A$-module $M\in {\mathcal D}_{(r-\epsilon, r+\epsilon)}$ iff $\phi'/\psi'$ is open on the $C$-module $(T,M) \in {\mathcal D}_{(\overline{\sigma}(r-\epsilon), \overline{\sigma}(r+\epsilon))}$.  The equivalence of the conditions therefore follows from \ref{cofinrtl} and \ref{allopenl1} and, for the equivalence of (iii) and (iv), \ref{allequiv}.
\end{proof}

If $A$ is a tubular algebra then denote by ${\rm Zg}^+_A$ the closed subset of the Ziegler spectrum consisting of all indecomposable pure-injectives with slope in $(0,\infty)$.  Let us consider the fibres of the slope map on ${\rm Zg}^+_A$, and say that a pp-pair $\phi/\psi$ is {\bf uniformly open}, respectively {\bf closed}, {\bf at} $r$ if $\phi/\psi$ is open, resp. closed, on every module of slope $r$.

\begin{cor}\label{phipsiopen}\marginpar{phipsiopen} Let $\phi/\psi$ be a pp-pair over a tubular algebra $A$.  Then for all but finitely many $r\in {\mathbb R}^+$, $\phi/\psi$ is either uniformly open at $r$ or uniformly closed at $r$.  Furthermore, the set of $r$ at which $\phi/\psi$ is uniformly open is the union of finitely many rational points and finitely many open intervals with rational endpoint(s).
\end{cor}
\begin{proof} Let $U$ be the set of all positive reals such that $\phi/\psi$ is open on each module with slope $r$. By \ref{allopenrtl} each irrational in $U$ is contained in an open interval contained in $U$. Furthermore, if we examine the proof of \ref{ppashom} then we see that there are only finitely many positive rationals $r$ (being slopes of indecomposable summands of the modules appearing) where the argument does not work; furthermore, there are only finitely many possible pointed modules $(M',m')$ produced by that proof (as $r$ varies).  If $r\in {\mathbb Q}^+$ is not one of the above exceptions then there is also a vector $v$ as in \ref{vdotdim} with ${\rm dim}(\phi/\psi)(X) = v\cdot\underline{\rm dim}(X)$ for all $X$ with slope in a neighbourhood of $r$.  Since that vector is defined in terms of the pointed modules involved, again there are only finitely many possibilities for $v$.  Given a non-exceptional rational $r$, if the vector $v$ as in \ref{vdotdim} is $0$, then $\phi/\psi$ is closed in an open neighbourhood of $r$.  We examine the case where $v\neq 0$.

For this part we again revert to the case of one of the algebras of the special forms $C$ used before.  We fix a non-exceptional rational $r$ and take $\epsilon$ and a vector $v\neq 0$ for measuring the dimension of $\phi/\psi$ as in \ref{vdotdim} (extended  to most rationals as discussed above).  Suppose that $X\in C\mbox{-}{\rm ind}$ has slope in $(r-\epsilon, r+\epsilon)$.  By \ref{bijomega} there are $c\in {\mathbb N}$, $\gamma\in {\mathbb Q}^\infty_0$ and $y\in \Omega \cup \{0\}$ such that $\underline{\rm dim}(X)= c(h_0+\gamma h_\infty)+y$ (and $c\gamma \in {\mathbb N})$.  The slope of $X$ therefore is $-\dfrac{c\gamma \langle h_0,h_\infty\rangle + \langle h_0,y\rangle}{c \langle h_\infty,h_0\rangle + \langle h_\infty,y\rangle}$ and, if $v\cdot \underline{\rm dim}(X)=0$, then we compute that $\gamma = \dfrac{-\frac{1}{c}v.y-v.h_0}{v.h_\infty}$.   By \ref{ineqomega} we may choose $\delta<\epsilon$ so that if the slope of $X$ as above is in $(r-\delta, r+\delta)$ then $\gamma = \iota(ch_0+c\gamma h_\infty)$ is in $(r-\epsilon, r+\epsilon)$.  We may assume that $r\neq -\dfrac{-v.h_0}{v.h_\infty}$ since this excludes only finitely many values of $r$ (as $v$ varies among the possible values discussed above) so, by choosing $\epsilon$ (and correspondingly $\delta$) small enough, we may assume that there are only finitely many values of $c$ such that $\dfrac{-\frac{1}{c}v.y-v.h_0}{v.h_\infty} \in (r-\epsilon, r+\epsilon)$, that is, only finitely many values of $\gamma$ corresponding to dimension vectors, $\underline{\rm dim}(X) = c(h_0+\gamma h_\infty)+y$ of $X$ with slope in $(r-\delta, r+\delta)$ such that $v.{\rm dim}(X)=0$.  Therefore, by choosing $\epsilon'$ small enough, we find a neighbourhood $(r-\epsilon', r+\epsilon')$ of $r$ such that if the slope of $X$ is in this interval then $v\cdot \underline{\rm dim}(X)\neq 0$ and hence $\phi(X)/\psi(X)\neq 0$.  That is, apart from finitely many values of $r\in {\mathbb Q}^+$, if $\phi/\psi$ is open at some module of slope $r$ then $\phi/\psi$ is open on every module with slope in an open neighbourhood of $r$.  The statement of the result then follows for these algebras $C$ and so, by using arguments as in the proof of \ref{allopenrtl}, for arbitrary tubular algebras.

\end{proof}

\end{document}